\documentclass[12pt,leqno,a4paper]{article}
\usepackage{color}

\usepackage[pdfborder={0 0 0}]{hyperref}

\usepackage{amssymb}
\usepackage{amsmath}
\usepackage{amsthm}

\usepackage[shortlabels]{enumitem}

\newtheorem{theorem}{Theorem}
\numberwithin{theorem}{section}
\newtheorem{proposition}[theorem]{Proposition}
\newtheorem{lemma}[theorem]{Lemma}
\newtheorem{corollary}[theorem]{Corollary}

\theoremstyle{remark}

\theoremstyle{definition}
\newtheorem{remark}[theorem]{Remark}
\newtheorem{definition}[theorem]{Definition}
\newtheorem{example}[theorem]{Example}

\numberwithin{equation}{section}

\newcommand\set[1]{\left\{\,#1\,\right\}}		
\newcommand\abs[1]{\left|#1\right|}				
\newcommand\ska[1]{\left\langle#1\right\rangle} 

\DeclareMathOperator{\id}{id}					
\DeclareMathOperator{\diag}{diag}				
\DeclareMathOperator{\tr}{tr}					
\DeclareMathOperator{\curl}{curl}				

\def\Z{\mathbb{Z}}
\def\Q{\mathbb{Q}}
\def\R{\mathbb{R}}
\def\C{\mathbb{C}}

\def\DC{\text{DC}}

\newcommand{\cC}{{\mathcal C}}

\newcommand{\cF}{{\mathcal F}}

\newcommand{\cL}{{\mathcal L}}

\begin{document}
\setcounter{footnote}{1}
\title{\texorpdfstring{Stability of periodic solutions of the $N$-vortex problem in general domains}{Stability of periodic solutions of the N-vortex problem in general domains}}
\author{Bj\"orn Gebhard\footnote{Supported by DAAD grant 57314604} \and Rafael Ortega\footnote{Supported by MTM2017-82348-C2-1-P (Spain)}}
\date{}
\maketitle

\begin{abstract}
We investigate stability properties of a type of periodic solutions of the $N$-vortex problem on general domains $\Omega\subset \mathbb{R}^2$. The solutions in question bifurcate from rigidly rotating configurations of the whole-plane vortex system and a critical point $a_0\in\Omega$ of the Robin function associated to the Dirichlet Laplacian of $\Omega$. Under a linear stability condition on the initial rotating configuration, which can be verified for examples consisting of up to 4 vortices, we show that the linear stability of the induced solutions is solely determined by the type of the critical point $a_0$. If $a_0$ is a saddle, they are unstable. Otherwise they are stable in a certain linear sense. The proof uses a criterion for the bifurcation of multiple eigenvalues, which is applied to suitable Poincar\'e sections. Beyond linear stability, Herman's last geometric theorem allows us to prove the existence of isoenergetically orbitally stable solutions in the case of $N=2$ vortices.
\end{abstract}



\section{Introduction}\label{sec:introduction}
The $N$-vortex problem is a first order Hamiltonian system describing a two-dimensional ideal fluid contained in a  domain $\Omega\subset\R^2$ solely by the positions of finitely many point vortices. In this model one supposes that the velocity field of the fluid $v:\Omega\times[0,T]\rightarrow\R^2$, $(x,t)\mapsto v(x,t)$ satisfying the 2D-Euler equations has a singular vorticity profile of the type
\[
\curl v(x,t)=\partial_{x_1}v_2(x,t)-\partial_{x_2}v_1(x,t)=\sum_{j=1}^N\Gamma_j \delta_{z_j(t)}(x),
\]
such that $z_j(t)$, $j=1,\ldots,N$ are the positions and $\Gamma_j\in\R\setminus\{0\}$, $j=1,\ldots,N$ the strengths of the point vortices.
Using the Euler equations one formally obtains that the time evolution of the positions is given by the Hamiltonian system
\begin{equation}\label{eq:n_vortex_intro}
\Gamma_j\dot{z}_j(t)=J\nabla_{z_j}H_\Omega(z_1(t),\ldots,z_N(t)),\quad j=1,\ldots,N,
\end{equation}
where $J\in\R^{2\times 2}$ denotes rotation by $-\frac{\pi}{2}$  and the Hamiltonian $H_\Omega$ defined on
\[
\cF_N(\Omega)=\set{z=(z_1,\ldots,z_N)\in\Omega^N:z_j\neq z_k~\text{for }k\neq j}
\]
reads
\[
H_\Omega(z_1,\ldots,z_N)=-\frac{1}{2\pi}\sum_{\underset{j\neq k}{j,k=1}}^N\Gamma_j\Gamma_k\log\abs{z_j-z_k}-\sum_{j,k=1}^N\Gamma_j\Gamma_k g_\Omega(z_j,z_k),
\]
with $g_\Omega:\Omega\times\Omega\rightarrow\R$ being the regular part of the Dirichlet Green's function
\[
G_\Omega(x,y)=-\frac{1}{2\pi}\log\abs{x-y}-g_\Omega(x,y). 
\]
I.e., for any $y\in \Omega$ the map $g_\Omega(\cdot,y)$ is harmonic in $\Omega$ and extends to $\overline{\Omega}$ with the boundary values $-\frac{1}{2\pi}\log\abs{\cdot-y}$. As a consequence there holds $g_\Omega(x,y)=g_\Omega(y,x)$ for any $(x,y)\in\Omega\times\Omega$. Note also that contrary to $G_\Omega$, the regular part $g_\Omega$ can be evaluated at the same point defining the so called Robin function $h_\Omega:\Omega\rightarrow \R$,
\[
h_\Omega(z)=g_\Omega(z,z).
\] 

Depending on the considered domain the derivation of the $N$-vortex problem \eqref{eq:n_vortex_intro} goes back to Kirchhoff \cite{kirchhoff_vorlesungen_1876}, Routh \cite{routh_applications_1880} and Lin \cite{lin_motion_1941,lin_motion_1941a}. For more modern literature treating this topic see \cite{flucher_variational_1999,marchioro_mathematical_1994,newton_n-vortex_2001,saffman_vortex_1993}. A rigorous justification of point vortex dynamics as a singular limit of the 2D-Euler equations can be found in \cite{davila_gluing_2018, marchioro_vortices_1993}.

We like to mention that $g_\Omega$ can also be the regular part of a hydrodynamic Green's function, which is a generalization of the Dirichlet Green's function, see \cite{flucher_variational_1999}, and that $N$-vortex type Hamiltonian systems also appear as a singular limit of the Gross-Pitaevskii equation and the Landau-Lifshitz-Gilbert equation, see \cite{jerrard_dynamics_1998,kurzke_ginzburglandau_2011} and references therein. In view of this, we will consider in our investigation a general symmetric and sufficiently smooth function $g:\Omega\times\Omega\rightarrow\R$ instead of $g_\Omega$.

As for Hamiltonian systems in general, a particular interest in the investigation of \eqref{eq:n_vortex_intro} lies on the examination of periodic solutions and their stability properties. In particular for the $N$-vortex system stability investigations can serve as a possible explanation for the lasting occurrence of vortex patterns in hurricanes, which have been observed both in numerical simulations \cite{kossin_mesovortices_2001} and in an actual hurricane \cite{kossin_mesovortices_2004}.

In the whole plane case $\Omega=\R^2$ the studies date back to the 19th century, when Thomson \cite{thomson_treatise_1883} has started the linear stability analysis of the regular $N$-Gon consisting of $N$ identical vortices. It took however more than a century until the nonlinear stability problem for the $N$-Gon, in particular for $N=7$, could be settled. It turned out that this configuration is stable (to be understood in the best possible sense), if and only if $N\leq 7$, see \cite{cabral_stability_2000} for a computer-aided proof and \cite{kurakin_stability_2002} for a non-computer-aided proof, as well as a nice historic review of the problem.

Also other rigidly rotating configurations in different settings have been investigated regarding their stability. Examples include the $N+1$-Gon \cite{cabral_stability_2000} and general configurations with different vorticities \cite{roberts_stability_2013} on the plane, the $N$-Gon in the unit disc \cite{kurakin_stability_2012}, the $N$-Gon and other configurations on the sphere \cite{boatto_nonlinear_2003,pekarsky_point_1998} and the $N$-Gon on a general surface with constant curvature \cite{boatto_curvature_2008}. Further articles regarding stability investigations of vortex configurations can be found in the given references. For an overview of possible relative equilibria solutions on $\R^2$ besides their stability, we refer to \cite{aref_relative_2011,aref_vortex_2003}.

All these settings above share the advantages that the Hamiltonian is explicitly known and invariant with respect to mutual rotations. For example on $\R^2$ the rotational symmetry allows a correspondence between relative equilibria and critical points of the Hamiltonian $H_{\R^2}$ constraint to a level set of the vortex angular impulse $I(z)=\sum_{j=1}^N\Gamma_j\abs{z_j}^2$. This way Dirichlet's criterion can be applied to obtain nonlinear stability results, see for example \cite{roberts_stability_2013}, and Morse theory can be used to investigate linear stability \cite{roberts_morse_2018}.

In general domains $\Omega\subset\R^2$, the Hamiltonian $H_\Omega$ does not share the advantages of the cases stated above. Still, using perturbative approaches the existence of some types of periodic solutions could also be established in general domains \cite{bartsch_periodic_2016a,bartsch_periodic_2018,bartsch_global_2017,bartsch_periodic_2018a,gebhard_periodic_2017}. In recent years, also different kinds of stationary solutions, i.e. critical points of $H_\Omega$,  could be shown to exist for a general $\Omega$, see \cite{bartsch_critical_2015,kuhl_equilibria_2016} and references therein. The question regarding the stability of these equilibria is open.

In this paper we will study stability properties of periodic solutions, that can be shown to exist by scaling the point vortices towards a critical point $a_0$ of the Robin function $h_\Omega$, i.e., towards a stationary solution of the $1$-vortex problem. This way the influence of the domain can be seen as a perturbation of the whole-plane case. In rescaled coordinates the solutions then bifurcate from relative equilibria of the whole-plane system, see Theorem \ref{thm:existence_result} and Section \ref{subsec:the_existence_result_revised} for details.    

Starting with a suitably stable relative equilibrium we will investigate the influence of the domain $\Omega$ on the stability of the bifurcated solutions. In fact we will mainly study linear stability properties in terms of the bifurcation of the Floquet multiplier $1$, which for a rigidly rotating configuration has at least multiplicity $4$. It turns out that only the type of the critical point $a_0$ of $h_\Omega$ influences the bifurcation of the multiplier. If $a_0$ is a saddle, the induced solutions are unstable. Otherwise they remain stable in a certain linear sense.

Examples of suitable rigidly rotating configurations include vortex pairs, equilateral triangles and rhombus configurations. Following ideas of \cite{ortega_stability_2017} based on Herman's last geometric theorem we can also conclude a nonlinear stability result for $N=2$ vortices.
As a consequence, the $2$-vortex problem in a generic bounded domain and with vorticities satisfying $\Gamma_1+\Gamma_2\neq 0$ always has isoenergetically orbitally stable periodic solutions with arbitrary small periods.

The precise statements Theorem \ref{thm:linear_stability}, Theorem \ref{thm:nonlinear_stability_2_vortices} and Corollary \ref{cor:two_vortex_problem_in_domains} can be found in Section \ref{sec:statement_of_results}. Section \ref{sec:preliminaries} contains necessary information about rigidly rotating configurations, specific examples, a revision of the existence result, as well as an expansion of the monodromy operator associated to the induced solutions. 

Based on this expansion we will determine the bifurcation of the Floquet multiplier $1$. One key tool here is a sufficient condition for the bifurcation of multiple eigenvalues that are simple in a higher order approximation, see Section \ref{sec:approximately_simple_eigenvalues} and in particular Lemma \ref{lem:actual_eigenvalues}. However, this criterion can not be applied in a direct way to the monodromy operator.
We therefore introduce in Section \ref{sec:application_to_a_poincare_section} suitable Poincar\'e sections of codimension $2$ and express them in a carefully choosen system of coordinates, such that the criterion applies after a not too large amount of calculations.
Finally, Section \ref{sec:nonlinear_stability} contains the proof of the nonlinear stability result for two vortices.

\section{Statement of results}\label{sec:statement_of_results}
Let $\Omega\subset \R^2$ be a domain, $\Gamma_1,\ldots,\Gamma_N\in\R\setminus\{0\}$, $g\in\cC^m(\Omega\times\Omega)$, $m\geq 2$
with $g(x,y)=g(y,x)$ for every $x,y\in \Omega$. As a generalization of the Dirichlet Green's and Robin function we define 
\begin{align*}
G(x,y)&=-\frac{1}{2\pi}\log\abs{x-y}-g(x,y),\\
h(x)&=g(x,x),
\end{align*}
as well as the corresponding $N$-vortex type Hamiltonian $H:\cF_N(\Omega)\rightarrow\R$,
\begin{align*}
H(z_1,\ldots,z_N)=\sum_{\underset{j\neq k}{j,k=1}}^N\Gamma_j\Gamma_kG(z_k,z_j)-\sum_{j=1}^N\Gamma_j^2h(z_j).
\end{align*}
In order to write \eqref{eq:n_vortex_intro} in a more compact way, we define the $2N\times 2N$ matrices $M_{\vec{\Gamma}}=\diag\big(\Gamma_1,\Gamma_1,\ldots,\Gamma_N,\Gamma_N\big)$ and $J_N=\diag\big(J,\ldots,J\big)$. Furthermore, the total vorticity and the total vortex angular momentum are denoted by 
\[
\Gamma=\sum_{j=1}^N\Gamma_j,\quad L=\sum_{j<k}\Gamma_j\Gamma_k.
\]
We will investigate stability properties of periodic solutions of the $N$-vortex type problem
\begin{equation}\label{eq:n_vortex_type_system}
M_{\vec{\Gamma}}\dot{z}=J_N\nabla H(z),
\end{equation}
which emanate from a critical point of $h$ and a relative equilibrium configuration of the whole-plane system
\begin{equation}\label{eq:n_vortex_on_whole_plane_1}
M_{\vec{\Gamma}}\dot{z}=J_N\nabla H_{\R^2}(z),
\end{equation}
see Theorem \ref{thm:existence_result} below. 

Let $Z(t)=e^{-\nu J_Nt}z_0$ be such a rigidly rotating solution of \eqref{eq:n_vortex_on_whole_plane_1}. After rescaling we can assume that $\nu=\pm 1$, i.e., $Z$ is $2\pi$-periodic. Since the Hamiltonian $H_{\R^2}$ is invariant under mutual translations and rotations, the linearized system
\begin{align}\label{eq:linearized_system_on_whole_plane}
M_{\vec{\Gamma}} \dot{v}=J_N\nabla^2 H_{\R^2}(Z(t))v
\end{align}
has at least $3$ linearly independent $2\pi$-periodic solutions, cf. Section \ref{subsec:the_whole_plane_case} for more details. Here $\nabla^2H_{\R^2}$ denotes the Hessian of $H_{\R^2}$. In terms of Floquet multipliers - the eigenvalues of $X_0(2\pi)$, where $X_0(t)\in\R^{2N\times 2N}$ satisfies \eqref{eq:linearized_system_on_whole_plane} with initial condition $X_0(0)=\id_{\R^{2N}}$ - this fact means that $Z$ has the multiplier $1$ with geometric multiplicity at least $3$. The configuration $Z(t)$ is called nondegenerate, if the geometric multiplicity of the multiplier $1$ is exactly $3$, or in other words, if the space of $2\pi$-periodic solutions of  \eqref{eq:linearized_system_on_whole_plane} has dimension $3$. 

Note that equations \eqref{eq:n_vortex_type_system}--\eqref{eq:linearized_system_on_whole_plane} are Hamiltonian with respect to the symplectic form
\begin{equation}\label{eq:definition_of_omega_Gamma}
\omega_{\vec{\Gamma}}(v,w)=\ska{M_{\vec{\Gamma}} v,J_Nw}_{\R^{2N}},
\end{equation}
thus the Floquet multipliers always appear in pairs $\lambda,\lambda^{-1}$ and the multiplier $1$ has even algebraic multiplicity, see \cite{meyer_introduction_2009}.

With the notion of a nondegenerate relative equilibrium the aforementioned existence result reads as follows.
\begin{theorem}\label{thm:existence_result}
If $\Gamma\neq 0$, $a_0\in\Omega$ is a nondegenerate critical point of $h$ and $Z(t)$ is a $2\pi$-periodic nondegenerate relative equilibrium solution of \eqref{eq:n_vortex_on_whole_plane_1}, then there exists a local family $\big(z^{(r)}\big)_{r\in(0,r_0)}$ of periodic solutions of \eqref{eq:n_vortex_type_system} having periods $2\pi r^2$ and satisfying
\[
z^{(r)}(t)=a_0+rZ(t/r^2)+o(r)
\] 
uniformly in $t$ as $r\rightarrow 0$. Moreover, the map
$
(0,r_0)\times\R\ni(r,t)\mapsto z^{(r)}(t)\in\cF_N(\Omega)
$
is of class $\cC^{m-1}$.
\end{theorem} 
This statement can be found in \cite{bartsch_global_2017}, Theorem 2.1 e) and also follows as a special case of \cite{gebhard_periodic_2017}, Theorem 1.9. A refinement of Theorem \ref{thm:existence_result} applies also to choreographic configurations, that are degenerate as a general relative equilibrium, but nondegenerate as a choreographic solution. An example is the Thomson $N$-Gon configuration. In fact the very first existence result for these kind of periodic solutions, due to Bartsch and Dai \cite{bartsch_periodic_2016a}, treats the regular $N$-Gon.

Concerning the existence of nondegenerate critical points of the actual Robin function $h_\Omega$ it has been shown in \cite{bartsch_morse_toappear} that at least after an arbitrary small deformation of the domain all critical points are nondegenerate. A minimum always exists, if $\Omega$ is bounded. 

We would also like to mention that for the existence alone, $a_0$ does not need to be a nondegenerate critical point, it is enough that $a_0$ is topological stable, i.e., that $a_0$ is an isolated critical point with nonvanishing Brouwer index, see \cite{bartsch_periodic_2016a,bartsch_global_2017}. But under these weaker assumption it is not clear, if the induced solutions form a continuous family of solutions, which is needed for our further stability analysis.

\begin{definition}\label{def:types_of_periodic_solutions} a) A periodic solution $z$ of a Hamiltonian system is called spectrally stable, if all Floquet multipliers lie on the unit circle $S^1$. Otherwise we say that it is spectrally unstable. The periodic solution is called L-stable, if it is spectrally stable, the multiplier $1$ has algebraic mutliplicity $2$ and the remaining nontrivial multipliers are all simple, in particular $-1$ is not contained in the Floquet spectrum.

b) Recall further that a relative equilibrium solution $Z(t)=e^{-\nu J_N t}z_0$ of \eqref{eq:n_vortex_on_whole_plane_1} is called nondegenerate, if the geometric multiplicity of the multiplier $1$ is $3$. If in addition the algebraic multiplicity is $4$, the configuration is called algebraic nondegenerate. The notion of L-stability is adapted for relative equilibria in the sense that $Z(t)$ is required to be spectrally stable, algebraic nondegenerate and the remaining nontrivial multipliers have to be simple. Again they are in particular different from $-1$. Relative equilibria satisfying these three properties are called LRE-stable.
\end{definition}

We will see that an algebraic nondegenerate relative equilibrium is also nondegenerate, cf. Lemma \ref{lem:floquet_solutions_of_whole_plane_case}. Furthermore, $\Gamma\neq 0$ and $L\neq 0$ are necessary conditions for algebraic nondegenerateness, see Lemma \ref{lem:algebraic_nondegenerate_implies_L_not_0}.

Regarding stability properties we start with the following observation, which is a direct consequence of the continuity of the Floquet multipliers and the equivalent rescaled formulation of Theorem \ref{thm:existence_result} in Proposition \ref{prop:existence_result_detailled} below.
\begin{remark}
If $Z(t)$ is already a spectrally unstable configuration, then also the induced solutions $z^{(r)}$ for $r>0$ small enough are spectrally unstable.
\end{remark}
Therefore it remains to look at spectrally stable configurations.
\begin{theorem}\label{thm:linear_stability} 
Let $g\in\cC^4(\Omega\times\Omega)$, $a_0\in\Omega$ be a nondegenerate critical point of $h$ and $Z(t)$ a $2\pi$-periodic LRE-stable relative equilibrium of \eqref{eq:n_vortex_on_whole_plane_1}.
\begin{enumerate}[a)]
\item If $a_0$ is a saddle point of $h$, then the induced solutions $z^{(r)}$ of \eqref{eq:n_vortex_type_system} are spectrally unstable for $r>0$ small enough.
\item If $h$ has a local minimum or maximum in $a_0$, then the induced solutions $z^{(r)}$ of \eqref{eq:n_vortex_type_system} are L-stable for $r>0$ small enough.
\end{enumerate}
\end{theorem}

The stability result applies to the following configurations, which have been investigated by Roberts \cite{roberts_stability_2013}. The details can be found in Section \ref{subse:examples}.
\begin{example}\label{ex:configurations} a) As long as $\Gamma\neq 0$, every solution of the $2$-vortex problem on $\R^2$ is a LRE-stable relative equilibrium solution.

b) Every equilateral triangle is a rigidly rotating configuration of the $3$-vortex problem provided $\Gamma\neq 0$. 
The triangle is algebraic degenerate in exactly two situations, if the total vortex angular momentum $L$ is vanishing or if all three vorticities are identical. If $L<0$, the triangle is spectrally unstable. Moreover, it is LRE-stable provided $L$ is positive, at least two vorticities are different and $10L\neq \Gamma_1^2+\Gamma_2^2+\Gamma_3^2$. 

c) Let $y>\frac{1}{\sqrt{3}}$ and $\kappa(y)=\frac{3y^2-y^4}{3y^2-1}$. Four vortices with stregths $\Gamma_1=\Gamma_2=1$, $\Gamma_3=\Gamma_4=\kappa(y)$ placed at $z_1=\pi^{-\frac{1}{2}}(-1,0)$, $z_2=\pi^{-\frac{1}{2}}(1,0)$, $z_3=\pi^{-\frac{1}{2}}(0,-y)$, $z_4=\pi^{-\frac{1}{2}}(0,y)$ form a rigidly rotating configuration. There exists $\delta>0$, such that this rhombus configuration is LRE-stable for $y\in(1-\delta,1+\delta)\setminus \{1\}$.
\end{example}
Contrary to these three positive examples, the straight line configurations for $N\geq 3$ are all nondegenerate, but spectrally unstable, see Corollary 3.3 in \cite{roberts_stability_2013}. Hence the induced solutions are spectrally unstable as well.

In the case that $\Omega$ is the unit disc, for any $\rho\in(0,1)$ the $N$-Gon placed on the circle of radius $\rho$ still forms a rigidly rotating configuration, whose stability has been studied. In particular the local part of the $N$-Gon family with $\rho>0$ small is stable if and only if $N\leq 6$, see \cite{kurakin_stability_2012} and references therein. This part of the family can be seen as a concrete example of Theorem \ref{thm:existence_result}. Now in a general domain $\Omega$ for the corresponding family induced by the Thomson $N$-Gon near a nondegenerate minimum of $h_\Omega$, it would therefore be interesting to see to what extend the stability properties coincide with those of the unit disc families. This however is not covered by our investigation. The whole-plane $N$-Gon is algebraically degenerate, since the algebraic multiplicity of the multiplier $1$ is at least $6$, see \cite{cabral_stability_2000}.

Regarding nonlinear stability we are able to state the following result in the case of $N=2$ vortices. Recall that in the classic case $g=g_\Omega$ is the regular part of the Dirichlet Green's function and in particular of class $\cC^\infty$.
\begin{theorem}\label{thm:nonlinear_stability_2_vortices}
Let $N=2$, $\Gamma_1,\Gamma_2\in\R\setminus\{0\}$ with $\Gamma\neq 0$, $g\in\cC^\infty(\Omega\times\Omega)$ and $a_0\in\Omega$ be a nondegenerate local minimum or maximum of $h$. Denote by $\big(z^{(r)}\big)_{r\in(0,r_0)}$ the periodic solutions of \eqref{eq:n_vortex_type_system} induced near $a_0$ by the $2\pi$-periodic vortex pair solution of the corresponding whole-plane system. Then there exists $r_1>0$, such that for almost every $r\in(0,r_1)$ the solution $z^{(r)}$ is isoenergetically orbitally stable.
\end{theorem} 

By isoenergetic orbital stability we mean that for any neighborhood $U$ of the orbit $z^{(r)}(\R)$ there exists another neighborhood $V$ of $z^{(r)}(\R)$, such that every solution $\xi(t)$ of \eqref{eq:n_vortex_type_system} with initial condition $\xi_0\in V\cap H^{-1}\big(z^{(r)}(0)\big)$ is defined for all $t\in\R$ and satisfies $\xi(\R)\subset U$. This is equivalent to saying that every symplectic Poincar\'e map associated to $z^{(r)}$ at some point $z^{(r)}(t_0)$ has this point as a Lyapunov stable fixed point. More on this notion of stability can be found in \cite{siegel_lectures_1971}.

We conclude this section with a consequence of Theorems \ref{thm:linear_stability}, \ref{thm:nonlinear_stability_2_vortices} and the fact that the Robin function $h_\Omega$ of a generic bounded domain has a nondegenerate minimum, \cite{bartsch_morse_toappear}. Here generic is understood in the sense that if $\Omega$ is a bounded domain with $\cC^{2,\alpha}$ boundary for some $0<\alpha<1$, then there exists an arbitrary small deformation $\psi\in\cC^{2,\alpha}(\overline{\Omega},\R^2)$, such that the Robin function associated to $\Omega_\psi=(\id_{\R^2}+\psi)(\Omega)$ has only nondegenerate critical points.  
\begin{corollary}\label{cor:two_vortex_problem_in_domains}
In a generic bounded domain the $2$-vortex problem with vorticities $\Gamma_1+\Gamma_2\neq 0$ always has a smooth family $\big(z^{(r)}\big)_{r\in(0,r_1)}$ of L-stable periodic solutions. Moreover, for almost every $r\in(0,r_1)$ the solution $z^{(r)}$ is isoenergetically orbitally stable.
\end{corollary}
\section{Preliminaries}\label{sec:preliminaries}
\subsection{Notation}
The euclidian scalar product on $\R^{2N}$ is denoted by $\ska{\cdot,\cdot}=\ska{\cdot,\cdot}_{\R^{2N}}$ and $V^\perp$ stands for the orthogonal complement of a subspace $V\leq \R^{2N}$ with respect to $\ska{\cdot,\cdot}$. 

Let $g:\Omega\times\Omega\rightarrow\R$, $(x,y)\mapsto g(x,y)$ and $F:\R^{2N}\rightarrow\R$, $z\mapsto F(z_1,\ldots,z_N)$ be smooth functions. The derivative $\nabla_1g(x,y)$ denotes the gradient of $g(\cdot,y)$ evaluated at $x$.
In the same way we understand the derivative $\nabla_{z_j}F$. Moreover, for $z,v_1,\ldots,v_{k-1}\in\R^{2N}$ the expression $\nabla^kF(z)[v_1,\ldots,v_{k-1}]$ denotes the unique vector $w\in \R^{2N}$ satisfying $\ska{w,\cdot}=D^kF(z)[v_1,\ldots,v_{k-1},\cdot]$. Similarly we write $\nabla^k F(z)[v_1,\ldots,v_{k-2}]$ for the unique symmetric matrix $W\in \R^{2N\times 2N}$ satisfying $\ska{W\cdot,\cdot}=D^kF(z)[v_1,\ldots,v_{k-2},\cdot,\cdot]$.

For $\R\times\R^{2N}\ni(r,z)\mapsto H_r(z)\in\R$ smooth we use the notation 
$\partial_r^j\nabla^kH_0(z)[\ldots]$ for the $j$th derivative $\left(\frac{d}{dr}\right)^j\big(\nabla^kH_r(z)[\ldots]\big)$ evaluated at $r=0$.

\subsection{The whole-plane case}\label{subsec:the_whole_plane_case}
We begin our investigation with a relative equilibrium solution $Z(t)=e^{-\nu J_N t}z_0$, $\nu\in\{\pm 1\}$, of the whole-plane system
\begin{equation}\label{eq:n_vortex_whole_plane}
M_{\vec{\Gamma}}\dot{z}=J_N\nabla H_0(z),
\end{equation}
where $H_0:\cF_N(\R^2)\rightarrow\R$,
\[
H_0(z)=H_{\R^2}(z)=-\frac{1}{2\pi}\sum_{\underset{k\neq j}{k,j=1}}^N\Gamma_k\Gamma_j\log\abs{z_j-z_k}.
\]
One immediatly sees that $H_0$ is invariant with respect to mutual translations, i.e., $H_0(\cdot+\hat{a})=H_0$ for every element $\hat{a}$ of
\[
D=\set{\hat{a}=(a,\ldots,a)\in\R^{2N}:a\in\R^2}.
\]
The space $D$ will play an important role in the stability analysis. For now we collect some basic properties of the Hamiltonian $H_0$ and the rigidly rotating solution $Z(t)$.
\begin{lemma}\label{lem:properties_of_whole_plane_hamiltonian}
For all $z\in\cF_N(\R^2)$, $t\in\R$, $k\geq 1$, $v_1,\ldots,v_{k-1}\in\R^{2N}$ there holds
\begin{gather}
\nabla^k H_0(z)[v_1,\ldots,v_{k-1}]\in D^\perp,\label{eq:lem:properties_of_whole_plane:orthogonality_of_H_0}\\
\ska{\nabla H_0(z),z}=-\frac{L}{\pi},\label{eq:lem:properties_of_whole_plane:nabla_H_times_z}\\
M_{\vec{\Gamma}} Z(t)\in D^\perp.\label{eq:lem:properties_of_whole_plane:center_of_vorticity_0}
\end{gather}
\end{lemma}
\begin{proof}
Property \eqref{eq:lem:properties_of_whole_plane:orthogonality_of_H_0} is a consequence of the invariance of $H_0$ with respect to mutual translations. Indeed $H_0(z+\hat{a})=H_0(z)$ for all $\hat{a}\in D$ implies $\nabla H_0(z)\in D^\perp$. Also the higher order derivatives are invariant under mutual translations, and therefore $D^kH_0(z)[v_1,\ldots,v_k]=0$ whenever one of the $v_j$ is contained in $D$. 

%

Equation \eqref{eq:lem:properties_of_whole_plane:nabla_H_times_z} follows by differentiation of
\[
H_0(\lambda z)=H_0(z)-\frac{L}{\pi}\log \lambda
\]
with respect to $\lambda$ at $\lambda=1$.

For the third property we use \eqref{eq:lem:properties_of_whole_plane:orthogonality_of_H_0} and the special shape of the solution $Z(t)=e^{-\nu J_N t}z_0$ with $\nu\in\{\pm 1\}$ to conclude
\begin{equation}\label{eq:identity_M_Gamma_z0_gradient_H0}
M_{\vec{\Gamma}} Z(t)=\nu J_NM_{\vec{\Gamma}}\dot{Z}(t)=-\nu \nabla H_0(Z(t))\in D^\perp.
\end{equation}
\end{proof}
Equation \eqref{eq:lem:properties_of_whole_plane:center_of_vorticity_0} is just one way of saying that the center of vorticity
\[
c_\Gamma=\frac{1}{\Gamma}\sum_{j=1}^N\Gamma_j Z_j(t),
\]
which is conserved along general solutions of \eqref{eq:n_vortex_whole_plane}, vanishes.

Now we look at the linearization of \eqref{eq:n_vortex_whole_plane} along $Z(t)$, which can be written as
\begin{equation}\label{eq:linearization_of_whole_plane_with_matrix_A}
\dot{v}=M_{\vec{\Gamma}}^{-1}J_N\nabla^2H_0(Z(t))v=:A_0(t)v.
\end{equation}
Let $X_0(t)\in\R^{2N\times 2N}$ denote the fundamental solution of \eqref{eq:linearization_of_whole_plane_with_matrix_A}, i.e.,
\[
\dot{X}_0(t)=A_0(t)X_0(t),\quad X_0(0)=\id_{\R^{2N}}.
\]
\begin{lemma}\label{lem:floquet_solutions_of_whole_plane_case}
The functions $\hat{a}\in D$, $\dot{Z}(t)$ and $Z(t)-2t\dot{Z}(t)$ solve the linearized equation \eqref{eq:linearization_of_whole_plane_with_matrix_A}.
As a consequence for every $\hat{a}\in D$, $t\in \R$ there holds
\begin{gather*}
X_0(t)\hat{a}=\hat{a},\quad X_0(t)J_Nz_0=e^{-\nu J_Nt}J_N z_0,\quad X_0(t)z_0=e^{-\nu J_Nt}(z_0+2t\nu J_Nz_0).
\end{gather*}
\end{lemma}
\begin{proof}
By \eqref{eq:lem:properties_of_whole_plane:orthogonality_of_H_0} the constant functions $\hat{a}\in D$ are solutions of \eqref{eq:linearization_of_whole_plane_with_matrix_A}. The derivative $\dot{Z}(t)$ is the canonical Floquet solution and $Z(t)-2t\dot{Z}(t)$ is a solution, since \eqref{eq:n_vortex_whole_plane} is invariant with respect to the scaling $z(t)\rightarrow r z(t/r^2)$. More precisely, with $Z(t)$ also every $Z_r(t)=r Z(t/r^2)$ is a solution of \eqref{eq:n_vortex_whole_plane}. Therefore 
\[
\frac{d}{d_r}_{|r=1}Z_r(t)=Z(t)-2t\dot{Z}(t)
\]
is a solution of the linearized equation. 
\end{proof}
The eigenvalues of the monodromy operator $X_0(2\pi)$ are the Floquet multipliers associated to $Z(t)$. Lemma \ref{lem:floquet_solutions_of_whole_plane_case} shows that elements of $D\oplus \R J_Nz_0$ are eigenvectors for the multiplier $1$. Note that $J_Nz_0$ is not contained in $D$, since $z_0\in \cF_N(\R^2)$. The nondegeneracy condition introduced in Definition \ref{def:types_of_periodic_solutions} says that the eigenspace $\ker (X_0(2\pi)-\id_{\R^{2N}})$ is exactly given by $D\oplus \R J_Nz_0$. By Lemma \ref{lem:floquet_solutions_of_whole_plane_case} we also know $(X_0(2\pi)-\id_{\R^{2N}})z_0=4\pi\nu J_Nz_0$ and thus $(X_0(2\pi)-\id_{\R^{2N}})^2z_0=0$. Hence $z_0$ is contained in the generalized eigenspace associated to the multiplier $1$ without being an actual eigenvector. 
This shows that every algebraic nondegenerate relative equilibrium solution is a nondegenerate relative equilibrium in the sense of Definition \ref{def:types_of_periodic_solutions} b).

\begin{remark}\label{rem:transpose_of_monodromy_operator} Since equation \eqref{eq:linearization_of_whole_plane_with_matrix_A} is Hamiltonian with respect to $\omega_{\vec{\Gamma}}$, the matrices $X_0(t)$ are symplectic, i.e., $\omega_{\vec{\Gamma}}(X_0(t)\cdot,X_0(t)\cdot)=\omega_{\vec{\Gamma}}$. By the definition of $\omega_{\vec{\Gamma}}$ in \eqref{eq:definition_of_omega_Gamma} we conclude
\[
X_0(t)^T=-J_NM_{\vec{\Gamma}}X_0(t)^{-1}J_NM_{\vec{\Gamma}}^{-1}.
\]
Therefore Lemma \ref{lem:floquet_solutions_of_whole_plane_case} implies 
\[
X_0(2\pi)^TM_{\vec{\Gamma}}z_0=M_{\vec{\Gamma}}z_0,\quad X_0(2\pi)^TM_{\vec{\Gamma}}J_Nz_0=M_{\vec{\Gamma}}J_Nz_0+4\pi\nu M_{\vec{\Gamma}}z_0.
\]
\end{remark}

\begin{lemma}\label{lem:algebraic_nondegenerate_implies_L_not_0}
If $Z(t)$ is algebraic nondegenerate, then necessarily $L\neq 0$ and $\Gamma\neq 0$.
\end{lemma}
\begin{proof} 
We first show that $L\neq 0$.
For $N=2$, we have $L=\Gamma_1\Gamma_2\neq 0$. We therefore consider $N\geq 3$.

Let $V=\R z_0\oplus \R J_N z_0$ and $W=M_{\vec{\Gamma}}V$. We are going to express the monodromy operator $X_0(2\pi)$ according to the splitting $\R^{2N}=W\oplus W^\perp$. By Remark \ref{rem:transpose_of_monodromy_operator} the plane $W$ is invariant under $X_0(2\pi)^T$, which implies that $W^\perp$ is invariant under $X_0(2\pi)$. Thus with respect to the splitting $W\oplus W^\perp$ the monodromy operator $X_0(2\pi)$ is expressed by a block matrix
\[
\begin{pmatrix}
A & 0\\
B & C
\end{pmatrix},
\]
where $A\in\R^{2\times 2}$ and  $C$ is the restriction of $X_0(2\pi)$ to $W^\perp$.

We compute the matrix $A$ with respect to the basis $M_{\vec{\Gamma}}z_0,M_{\vec{\Gamma}}J_Nz_0$ of $W$. Note that this is a orthogonal basis due to the block diagonal structure of $M_{\vec{\Gamma}}$ and $J_N$. We write
\begin{align*}
X_0(2\pi)M_{\vec{\Gamma}}z_0&=\alpha M_{\vec{\Gamma}}z_0+\beta M_{\vec{\Gamma}}J_Nz_0+w_1,\\
X_0(2\pi)M_{\vec{\Gamma}}J_Nz_0&=\gamma M_{\vec{\Gamma}}z_0+\delta M_{\vec{\Gamma}}J_Nz_0+w_2
\end{align*} 
with $w_1,w_2\in W^\perp$. The coefficients $\alpha,\beta,\gamma,\delta$ can then be computed by taking inner products and the use of Remark \ref{rem:transpose_of_monodromy_operator}. For example
\begin{align*}
\beta \abs{M_{\vec{\Gamma}}J_Nz_0}^2&=\ska{X_0(2\pi)M_{\vec{\Gamma}}z_0,M_{\vec{\Gamma}}J_Nz_0}=\ska{M_{\vec{\Gamma}}z_0,X_0(2\pi)^TM_{\vec{\Gamma}}J_Nz_0}\\
&=\ska{M_{\vec{\Gamma}}z_0,M_{\vec{\Gamma}}J_Nz_0+4\pi\nu M_{\vec{\Gamma}}z_0}=4\pi\nu \abs{M_{\vec{\Gamma}}z_0}^2
\end{align*}
shows that $\beta=4\pi\nu$. In a similar way we get
\[
A=\begin{pmatrix}
\alpha & \gamma\\
\beta &\delta
\end{pmatrix}=\begin{pmatrix}
1 & 0\\
4\pi\nu & 1
\end{pmatrix}.
\]
Hence the multiplicity of the Floquet multiplier $1$ of $Z(t)$ is given by $2+\nu_1(C)$, where $\nu_1(C)$ denotes the multiplicity of 
the eigenvalue $1$ of $C$, i.e., of the restriction $X_0(2\pi)_{|W^\perp}$. By Lemma \ref{lem:floquet_solutions_of_whole_plane_case} we know that the $4$-dimensional space $D\oplus V$ is contained in the generalized eigenspace of $X_0(2\pi)$ associated to the eigenvalue $1$. The space $D$ is always contained in $W^\perp$, due to \eqref{eq:lem:properties_of_whole_plane:center_of_vorticity_0} and the invariance of $D$ under $J_N$. To finish the proof we will see that under the condition $L=0$ the inclusion $V\subset W^\perp$ holds as well. In consequence $\nu_1(C)\geq 4$, such that the multiplier $1$ of $Z(t)$ has at least multiplicity $6$. 

Indeed, if $L=0$, the identities \eqref{eq:lem:properties_of_whole_plane:nabla_H_times_z} and $M_{\vec{\Gamma}}z_0=-\nu \nabla H_0(z_0)$, cf. \eqref{eq:identity_M_Gamma_z0_gradient_H0}, imply
\begin{align*}
\ska{z_0,M_{\vec{\Gamma}}z_0}=\ska{z_0,-\nu \nabla H_0(z_0)}=\frac{\nu L}{\pi}=0.
\end{align*}
Regarding the other inner products we have
\[
\ska{J_Nz_0,M_{\vec{\Gamma}}J_Nz_0}=\ska{z_0,M_{\vec{\Gamma}}z_0}=0,\quad \ska{J_Nz_0,M_{\vec{\Gamma}}z_0}=-\ska{z_0,M_{\vec{\Gamma}}J_Nz_0}=0
\]
by the block diagonal structure of the matrices $M_{\vec{\Gamma}}$, $J_N$ and therefore $V\subset W^\perp$.

It remains to show that $\Gamma=0$ implies the algebraic degenerateness of $Z(t)$. Again we consider $N\geq 3$, since for $N=2$ the vortex pair with $\Gamma_1=1$, $\Gamma_2=-1$ is rigidly translating and not rotating. 

We use Lemma 2.6 of \cite{moser_notes_2005}, which says that if $\dot{x}=X_{H_1}(x)$ is a Hamiltonian system with $r$ integrals $H_1,\ldots,H_r$, such that $\nabla H_1,\ldots,\nabla H_r$ are linearly independent and $\set{H_\alpha,H_\beta}=0$ for $\alpha=1,\ldots,r$, $\beta=1,\ldots,s$, then any closed orbit of the system has the Floquet multiplier $1$ with multiplicity at least $r+s$. Here $\set{\cdot,\cdot}$ denotes the Poisson bracket associated to the symplectic form.

In our case, the system \eqref{eq:n_vortex_whole_plane} has besides the Hamiltonian $H_0(z)$ itself the two components of the center of vorticity
\[
Q(z)=\sum_{j=1}^N\Gamma_j x_j,\quad P(z)=\sum_{j=1}^N\Gamma_j y_j,\quad z_j=(x_j,y_j)
\]
as first integrals. At least locally near the orbit of $Z(t)$ the gradients are linearly independent, since $z_0\notin D$ and
\[
\nabla H_0(z_0)=-\nu M_{\vec{\Gamma}}z_0,\quad \nabla Q(z_0)=M_{\vec{\Gamma}}\hat{e}_1,\quad \nabla P(z_0)=M_{\vec{\Gamma}}\hat{e}_2,
\]
where $e_1=(1,0)^T$, $e_2=(0,1)^T$. If $F:\R^{2N}\rightarrow\R$ is a $C^1$ function, the symplectic vectorfield $X^{\vec{\Gamma}}_F$ induced by $\omega_{\vec{\Gamma}}$ reads $J_NM_{\vec{\Gamma}}^{-1}\nabla F$. Therefore the Poisson bracket is given by
\[
\set{F_1,F_2}_{\vec{\Gamma}}=\omega_{\vec{\Gamma}}\left(X^{\vec{\Gamma}}_{F_1},X^{\vec{\Gamma}}_{F_2}\right)=\ska{\nabla F_1,J_NM_{\vec{\Gamma}}^{-1}\nabla F_2}.
\]
By \eqref{eq:lem:properties_of_whole_plane:orthogonality_of_H_0} we immediately see $\set{H_0,Q}_{\vec{\Gamma}}=0$, $\set{H_0,P}_{\vec{\Gamma}}=0$ on all of $\cF_N(\R^2)$. Furthermore, $\set{Q,P}_{\vec{\Gamma}}=\Gamma=0$ by our assumption.
Hence the three integrals are in involution and the multiplier $1$ for $Z(t)$ has multiplicity greater or equal $6$.
\end{proof}
\begin{remark}
The quantities $\Gamma$ and $L$ can not vanish at the same time since $\Gamma^2=2L+\sum_{j=1}^N\Gamma_j^2$.

The equilateral triangle with vorticities $\Gamma_1=1$, $\Gamma_2=1$, $\Gamma_3=-\frac{1}{2}$ is an example for a configuration with $L=0$, cf. Example \ref{ex:configurations} b).

An example of a rotating configuration with $\Gamma=0$ can also be realised with $3$ vortices. Here one takes $\Gamma_1=1$, $\Gamma_2=\Gamma_3=-\frac{1}{2}$ at the positions $z_1=0$, $z_2=(\rho,0)$, $z_3=-z_2$, $\rho>0$, such that the two vortices with negative strengths rigidly rotate around the central vortex with a frequency depending on $\rho$.

Another example with $\Gamma =0$ is the Rhombus configuration with parameter value $y=\sqrt{3+\sqrt{8}}$, cf. Example \ref{ex:configurations} c). See also Section \ref{subse:examples} below.
\end{remark}

Different proofs of Lemma \ref{lem:floquet_solutions_of_whole_plane_case} and of the part concerning $L\neq 0$ in Lemma \ref{lem:algebraic_nondegenerate_implies_L_not_0} can also be found in the paper by Roberts \cite{roberts_stability_2013}.
\subsection{Examples}\label{subse:examples}
a) Let $N=2$, $\Gamma\neq 0$. Two vortices rigidly rotate around the center of vorticity with frenquency $\nu=\frac{\Gamma}{\pi D^2}$, where $D=\abs{Z_1(t)-Z_2(t)}$ is the distance between the vortices. By Lemma \ref{lem:floquet_solutions_of_whole_plane_case} we know the whole monodromy matrix in this situation. Since there is no other multiplier than $1$, vortex pairs are LRE-stable.

b) Let $N=3$, $\Gamma\neq 0$. Then every equilateral triangle is a rigidly rotating configuration. Via scaling it can be assumed that the frequency $\nu$ is given by $\nu=\frac{\Gamma}{3}$. Recovering the spectral stability result of Synge \cite{synge_motion_1949}, Roberts has shown in \cite{roberts_stability_2013} that for this frequency the nontrivial Floquet multipliers are $\lambda_\pm =e^{\pm\frac{2\pi}{\nu}\sqrt{\frac{-L}{3}}}$. For the stability result in Theorem \ref{thm:linear_stability} part c) we need $\lambda_\pm\in S^1\setminus\{\pm 1\}$. This implies $L>0$ and $\sqrt{\frac{L}{3}}\notin \frac{\nu}{2}\Z=\frac{\Gamma}{6}\Z$. If we suppose that 
\[
\frac{L}{3}=\frac{1}{36}\Gamma^2k^2=\frac{1}{36}(\Gamma_1^2+\Gamma_2^2+\Gamma_3^2+2L)k^2
\]
for some $k\in\Z$, then 
\[
2(6-k^2)L=(\Gamma_1^2+\Gamma_2^2+\Gamma_3^2)k^2
\]
implies $\abs{k}=1$ or $\abs{k}=2$. But then $10L=\Gamma_1^2+\Gamma_2^2+\Gamma_3^2$ or $L=\Gamma_1^2+\Gamma_2^2+\Gamma_3^2$. The triangle is therefore LRE-stable, when $L>0$ and none of the two equations hold. 

In a similar way one sees that the triangle is algebraic nondegenerate as long as $L\in\R\setminus\set{0, \Gamma_1^2+\Gamma_2^2+\Gamma_3^2}$. Note that the condition $L=\Gamma_1^2+\Gamma_2^2+\Gamma_3^2$ is equivalent to $\Gamma_1=\Gamma_2=\Gamma_3$. Indeed with $\alpha=(\Gamma_1,\Gamma_3,\Gamma_2)^T$, $\beta=(\Gamma_2,\Gamma_1,\Gamma_3)^T$ we get
\[
L=\ska{\alpha,\beta}\leq \abs{\alpha}\abs{\beta}=\Gamma_1^2+\Gamma_2^2+\Gamma_3^2.
\]
The inequality is strict unless $\alpha$ and $\beta$ are linearly dependent. Since $\abs{\alpha}=\abs{\beta}$, this implies $\alpha=\pm\beta$. In fact $\alpha=-\beta$ is not possible, hence $\alpha=\beta$ and $\Gamma_1=\Gamma_2=\Gamma_3$.

c) Let $N=4$ and $y>\frac{1}{\sqrt{3}}$. The rhombus configuration consisting of the vortices $\Gamma_1=\Gamma_2=1$ and $\Gamma_3=\Gamma_4=\kappa(y)=\frac{3y^2-y^4}{3y^2-1}$ at  $z_1=\pi^{-\frac{1}{2}}(-1,0)$, $z_2=\pi^{-\frac{1}{2}}(1,0)$, $z_3=\pi^{-\frac{1}{2}}(0,-y)$, $z_4=\pi^{-\frac{1}{2}}(0,y)$ rotates with frequency 
\[
\nu(y)=\frac{1}{2}+\frac{2\kappa(y)}{y^2+1}.
\]
The existence of this configuration has been shown in \cite{hampton_relative_2014}, its stability has been investigated in \cite{roberts_stability_2013}. Note that our Hamiltonian differs by a factor of $\pi^{-1}$ from the Hamiltonian used in \cite{hampton_relative_2014,roberts_stability_2013}, which is why we included the factor $\pi^{-\frac{1}{2}}$.
It follows from \cite{roberts_stability_2013} that besides the fourfold multiplier $1$ the four nontrivial Floquet multipliers are given by
\[
\lambda^\pm_j(y)=e^{\pm \frac{2\pi i}{\nu(y)}\sqrt{\nu(y)^2-\mu_j(y)^2}}, \quad j=1,2,
\]
where
\begin{align*}
\mu_1(y)&=\frac{7y^4-18y^2+7}{2(y^2+1)(3y^2-1)},\\
\mu_2(y)&=\frac{2(y^2-1)(y^2+2y-1)(y^2-2y-1)}{(y^2+1)^2(3y^2-1)}.
\end{align*}

For $y=1$ we get $\kappa(1)=1$ and recover the regular $4$-Gon rotating with frequency $\nu(1)=\frac{3}{2}$. The associated multipliers are 
\[
\lambda^\pm_1(1)=e^{\pm \frac{4\sqrt{2}\pi i}{3}},\quad \lambda^\pm_2(1)=1,
\]
which shows that the $4$-Gon configuration is degenerate. Note that the multiplier $\lambda_1^+(1)$ lies in the open circle segment between $e^{-\frac{\pi}{3}i}$ and $1$. Since the multipliers depend continuously on $y$, the same is true for $\lambda_1^+(y)$ with $y$ close to $1$.

Regarding the second pair of multipliers $\lambda_2^\pm(y)=\exp\left(\pm 2\pi i\sqrt{1-\frac{\mu_2(y)^2}{\nu(y)^2}}\right)$ we see that $\lambda_2^\pm(y)\neq 1$ for $1\neq y\approx 1$, since $\frac{\mu_2(y)}{\nu(y)}$ is a non-constant rational function vanishing at $y=1$.
 This way we find $\delta>0$, such that $\lambda_1^\pm(y),\lambda_2^\pm(y)\in S^1\setminus\{\pm 1\}$ and $\lambda_1^+(y)\neq \lambda_2^\pm(y)$ for $y\in(1-\delta,1+\delta)\setminus\{1\}$. In other words the rhombus configuration with $y\in(1-\delta,1+\delta)\setminus\{1\}$ is LRE-stable.

\subsection{The existence result revised}\label{subsec:the_existence_result_revised}
Let $g$ be of class $\cC^m$, $m\geq 2$ and $\Gamma_1,\ldots,\Gamma_N$, $a_0$ and $Z(t)$ be as in Theorem \ref{thm:existence_result}. For the critical point $a_0$ we assume $a_0=0$ throughout the rest of the paper.

A crucial step of the proof of Theorem \ref{thm:existence_result} is a rescaling of the problem. Let $F:\Omega^N\rightarrow\R$,
\[
F(z)=\sum_{j,k=1}^N\Gamma_k\Gamma_jg(z_j,z_k).
\] 
The ansatz $z(t)=ru(t/r^2)$, $r>0$ shows that $z$ solves system \eqref{eq:n_vortex_type_system}, if and only if $u$ is a solution of
\begin{equation}\label{eq:n_vortex_with_parameter}
M_{\vec{\Gamma}} \dot{u}=J_N\nabla H_r(u),
\end{equation}
where $H_r:\cF_N\big(\Omega/r\big)\rightarrow\R$,
\begin{align*}
H_r(u)&=-\frac{1}{2\pi}\sum_{\underset{k\neq j}{k,j=1}}^N\Gamma_k\Gamma_j\log\abs{u_j-u_k}-\sum_{j,k=1}^N\Gamma_k\Gamma_jg(ru_j,ru_k)+F(0)\\
&=H_0(u)-F(ru)+F(0).
\end{align*}
Thus the scaling provides a family of equivalent Hamiltonian systems extending to the whole-plane system \eqref{eq:n_vortex_whole_plane} as $r\rightarrow 0$. In fact the map 
\[ 
\set{(r,u)\in\R^{1+2N}:r\geq 0,~ru\in\Omega^N,~u\in\cF_N(\R^2)}\rightarrow\R,~(r,u)\mapsto H_r(u)
\]
is of class $\cC^{m}$.

By our assumption we know that $Z(t)$ is a $2\pi$-periodic solution of \eqref{eq:n_vortex_whole_plane}, which is supposed to be as nondegenerate as possible, i.e., the space of $2\pi$-periodic solutions of the linearization \eqref{eq:linearized_system_on_whole_plane} is exactly given by $D\oplus\R\dot{Z}$, cf. Section \ref{subsec:the_whole_plane_case}. It can be shown that the degeneracy given by the mutual translations $D$ can be compensated by the nondegenerate critical point $a_0=0$ of the generalized Robin function $h(x)=g(x,x)$. This way the $2\pi$-periodic solution $Z(t)$ of \eqref{eq:n_vortex_with_parameter} at $r=0$ can be continued to $2\pi$-periodic solutions of \eqref{eq:n_vortex_with_parameter} for $r>0$ small. The proof uses the associated variational structure on the Sobolev space $H^1(\R/2\pi\Z,\R^{2N})=H^1_{2\pi}$ of $2\pi$-periodic continuous functions having a square-integrable derivative.
\begin{proposition}\label{prop:existence_result_detailled}[\cite{gebhard_periodic_2017}, Theorem 1.9]
Under the conditions of Theorem \ref{thm:existence_result} there exists $r_0>0$ and a $\cC^{m-2}$-map $[0,r_0)\ni r\mapsto u^{(r)}\in H^1_{2\pi}$, such that $u^{(r)}$ solves \eqref{eq:n_vortex_with_parameter} and $u^{(0)}=Z$. Moreover, restricted to $(0,r_0)$ the map is of class $\cC^{m-1}$ and if $m\geq 3$, then 
\begin{equation}\label{eq:derivative_of_u_r}
\partial_ru^{(0)}=\frac{d}{dr}_{|r=0}u^{(r)}\in D.
\end{equation}
\end{proposition}

Since $H^1_{2\pi}$ embeds into the space of $2\pi$-periodic continuous functions and since the flow $(r,t,u_0)\mapsto\phi_r(t,u_0)$ associated to the parameter dependent family of Hamiltonian systems \eqref{eq:n_vortex_with_parameter} is of class $\cC^{m-1}$, we obtain that
\[
[0,r_0)\times\R\ni(r,t)\mapsto u^{(r)}(t)\in\R^{2N}
\]
is a $\cC^{m-2}$ mapping. Moreover, $\partial_r^{m-2}u^{(r)}$ is continuously differentiable with respect to $t$. Restricted to $(0,r_0)\times\R$ the map $(r,t)\mapsto u^{(r)}(t)$ is even of class $\cC^{m-1}$ and the solutions $\big(z^{(r)}\big)_{r\in(0,r_0)}$ of the unscaled equation \eqref{eq:n_vortex_type_system} in Theorem \ref{thm:existence_result} are given by $z^{(r)}(t)=ru^{(r)}(t/r^2)$.

The next Lemma shows that without restriction we can pass to a slightly time shifted version of the family $\big(u^{(r)}\big)_{r\in[0,r_0)}$. This will be useful in later computations.
\begin{lemma}\label{lem:admissible_time_shift}
If $\tau \in\cC^{m-2}([0,r_0))\cap\cC^{m-1}((0,r_0))$, $r\mapsto \tau_r$ with $\tau_0\in 2\pi\Z$ and  $\partial_r\tau_0=0$, then the time shifted family $\big(\tilde{u}^{(r)}\big)_{r\in[0,r_0)}$, $\tilde{u}^{(r)}=u^{(r)}(\cdot+\tau_r)$ has exactly the same properties as the original family $\big(u^{(r)}\big)_{r\in[0,r_0)}$ of Proposition \ref{prop:existence_result_detailled}.
\end{lemma}
\begin{proof}
Clearly $r\mapsto\tilde{u}^{(r)}$ has the same regularity  as $r\mapsto u^{(r)}$, $\tilde{u}^{(r)}$ solves \eqref{eq:n_vortex_with_parameter} and $\tilde{u}^{(0)}=Z$, $\partial_r\tilde{u}^{(0)}=\partial_r u^{(0)}\in D$.
\end{proof}

\subsection{The monodromy operator}\label{subsec:monodromy_operator}
Observe that the Floquet multipliers as well as further stability properties of $z^{(r)}$ as a $2\pi r^2$-periodic solution of \eqref{eq:n_vortex_type_system} are the same as the Floquet multipliers and stability properties of $u^{(r)}$ as a $2\pi$-periodic solution of \eqref{eq:n_vortex_with_parameter}. We therefore study for $r\in(0,r_0)$ the linearization 
\begin{equation}\label{eq:linearization_of_n_vort_with_parameter}
\dot{v}=M_{\vec{\Gamma}}^{-1}J_N\nabla^2H_r\big(u^{(r)}(t)\big)v=:A_r(t)v
\end{equation}
of \eqref{eq:n_vortex_with_parameter} along the $2\pi$-periodic solution $u^{(r)}$. Let $X_r(t)\in\R^{2N\times 2N}$, $r\in(0,r_0)$ denote the fundamental solution of \eqref{eq:linearization_of_n_vort_with_parameter}, i.e.,
\begin{equation}\label{eq:definition_of_fundamental_solution}
\dot{X}_r(t)=A_r(t)X_r(t),\quad X_r(0)=\id_{\R^{2N}}.
\end{equation}
 The definitions of $A_r$ and $X_r$ for $r>0$ fit the definition for $r=0$ in Section \ref{subsec:the_whole_plane_case} in the sense that $[0,r_0)\times\R\rightarrow \R^{2N\times 2N}$, $(r,t)\mapsto A_r(t)$ and $(r,t)\mapsto X_r(t)$ are $\cC^{m-2}$ maps. Additionally, $\partial_r^{m-2}X_r$ is continuously differentiable with respect to $t$ and satisfies a certain differential equation.

\begin{lemma}\label{lem:basic_properties_of_X_r_A_r}
Let $m\geq 4$. For $a\in\R^2$, $1\leq j<k$, $j+k\leq m$ there holds
\begin{gather}
\nabla F(\hat{a})=\Gamma M_{\vec{\Gamma}} \widehat{\nabla h(a)},\label{eq:lem:basic_properties_nabla_F}\\
\partial_r\nabla H_0=0,\quad\partial_r^2\nabla^2H_0(u)=-2\nabla^2F(0),\quad\partial_r^j\nabla^kH_0=0. \label{eq:lem:basic_properties_derivatives_of_H_r_0}
\end{gather}
\end{lemma}
\begin{proof}
By the symmetry of $g$ and the definition of $h(x)=g(x,x)$ we obtain 
\[
\nabla_{z_j}F(z)_{|z=\hat{a}}=\Gamma_j\sum_{k=1}^N\Gamma_k2\nabla_1g(z_j,z_k)_{|z=\hat{a}}=\Gamma_j\Gamma\nabla h(a).
\]
Thus we conclude $\nabla F(\hat{a})=\Gamma M_{\vec{\Gamma}}\widehat{\nabla h(a)}$. In particular $\nabla F(0)=0$. 

Next we have $\nabla H_r(u)=\nabla H_0(u)-r\nabla F(ru)$, so differentiation with respect to $r$ at $r=0$ shows 
\[
\partial_r\nabla H_0(u)=-\big(\nabla F(ru)+r\nabla^2F(ru)u\big)_{|r=0}=0.
\]
In a similar way one sees that $\partial_r^2\nabla^2H_0(u)=-2\nabla^2 F(0)$ and that $\partial_r^j\nabla^kH_0(u)=0$ whenever $j<k$.
\end{proof}
\begin{lemma}\label{lem:expansion_of_monodromy_operator_on_D}
Let $m\geq 4$. As $r\rightarrow 0$ the expansion
\[
X_r(t)=X_0(t)+\frac{1}{2}r^2 X_0(t)\int_0^tX_0(s)^{-1}\partial_r^2A_0(s)X_0(s)\:ds+o(r^2)
\]
holds pointwise in $t\in\R$. In particular for $a\in\R^2$ we have
\[
X_r(t)\hat{a}=\hat{a}-r^2\Gamma t\big(J\nabla^2 h(0)a\big)\widehat{\phantom{|}}+o(r^2).
\]
\end{lemma}
\begin{proof}
Recall that $r\mapsto X_r(t)$ is at least $\cC^2$ when $m\geq 4$, so we can expand $X_r(t)$ in second order. For the first derivative $\partial_rX_0(t)$ we use \eqref{eq:lem:properties_of_whole_plane:orthogonality_of_H_0}, \eqref{eq:lem:basic_properties_derivatives_of_H_r_0} and $\partial_ru^{(0)}\in D$, cf. \eqref{eq:derivative_of_u_r},  to conclude
\[
\partial_rA_0(t)=M_{\vec{\Gamma}}^{-1}J_N\big(\partial_r\nabla^2H_0(Z(t))+\nabla^3H_0(Z(t))[\partial_ru^{(0)}]\big)=0.
\]

Differentiation of \eqref{eq:definition_of_fundamental_solution} at $r=0$ shows that $\partial_rX_0(t)$ satisfies 
\[
\big(\partial_rX_0\big)\dot{\phantom{|}}(t)=A_0(t)\partial_rX_0(t)+\partial_rA_0(t)X_0(t),\quad \partial_rX_0(0)=0,
\]
but since $\partial_rA_0(t)=0$, $\partial_rX_0(t)=0$ follows. In a similar way we see that 
\[
\big(\partial^2_rX_0\big)\dot{\phantom{|}}(t)=A_0(t)\partial^2_rX_0(t)+\partial^2_rA_0(t)X_0(t),\quad \partial^2_rX_0(0)=0,
\]
such that variation of constants for $\partial_r^2X_0(t)$ leads to the stated general expansion. 

If we now in particular consider $X_r(t)\hat{a}$ with $\hat{a}\in D$, we already know by Lemma \ref{lem:floquet_solutions_of_whole_plane_case} that $X_0(t)\hat{a}=\hat{a}$. We also conclude
\[
\partial_r^2X_0(t)\hat{a}=X_0(t)\int_0^tX_0(s)^{-1}\partial_r^2A_0(s)\hat{a}\:ds
\]
and by \eqref{eq:lem:properties_of_whole_plane:orthogonality_of_H_0}, \eqref{eq:lem:basic_properties_nabla_F}, \eqref{eq:lem:basic_properties_derivatives_of_H_r_0} there holds
\begin{align*}
\partial_r^2A_0(s)\hat{a}&=M_{\vec{\Gamma}}^{-1}J_N\big(\partial_r^2\nabla^2 H_0(Z(s))+\nabla^3H_0(Z(s))[\partial_r^2u^{(0)}(s)]\big)\hat{a}\\
&=-M_{\vec{\Gamma}}^{-1}J_N2\nabla^2F(0)\hat{a}=-2\frac{d}{d\varepsilon}_{|\varepsilon=0}M_{\vec{\Gamma}}^{-1}J_N\nabla F(\varepsilon\hat{a})\\
&=-2\Gamma \frac{d}{d\varepsilon}_{|\varepsilon=0} \big(J\nabla h(\varepsilon a)\big)\widehat{\phantom{|}}=-2\Gamma \big(J\nabla^2 h(0)a\big)\widehat{\phantom{|}}\in D. 
\end{align*}
It follows $\partial_r^2X_0(t)\hat{a}=-2t\Gamma \big(J\nabla^2 h(0)a\big)\widehat{\phantom{|}}$ and therefore the desired expansion.
\end{proof}
As a last step in this section we consider the extension of the monodromy operator $X_r(2\pi)$ to a complex linear map $X_r(2\pi):\C^{2N}\rightarrow \C^{2N}$ and define the notation $\hat{a}$ also for complex vectors $\hat{a}=(a,\ldots,a)\in\C^{2N}$, $a\in \C^2$.
\begin{corollary}\label{cor:almost_eigenvalues_for_X_r} Let $m\geq 4$.
There exist $a_{\pm}\in\C^2\setminus\{0\}$, such that
\[
X_r(2\pi)\hat{a}_{\pm}=\left(1\pm 2\pi\Gamma r^2\sqrt{-\det\nabla^2h(0)}\right)\hat{a}_{\pm}+o(r^2).
\]
\end{corollary}
\begin{proof}
Let $\lambda_1,\lambda_2\in\R\setminus\{0\}$ denote the eigenvalues of $\nabla^2h(0)$ and $S\in\R^{2\times 2}$ be an orthogonal matrix with 
\[
S^T\nabla^2h(0)S=\begin{pmatrix}
\lambda_1 & 0\\
0 & \lambda_2
\end{pmatrix}.
\]
Without restriction we can assume $\det S=1$, which in dimension $2$ means that $S$ is a symplectic matrix, i.e. $S^TJS=J$. By Lemma \ref{lem:expansion_of_monodromy_operator_on_D} we need to investigate the spectrum of $J\nabla^2h(0)$. The characteristic polynomial reads
\begin{align*}
\det\left(J\nabla^2h(0)-\mu\id_{\C^2}\right)&=\det\left(SS^TJSS^T\nabla^2h(0)SS^T-\mu SS^T\right)\\
&=\det\left(J\begin{pmatrix}
\lambda_1&0\\
0&\lambda_2
\end{pmatrix}-\mu\id_{\C^2}\right)\\
&=\mu^2+\lambda_1\lambda_2.
\end{align*}
Thus the eigenvalues of $J\nabla^2h(0)$ are given by $\mu_{1,2}=\pm\sqrt{-\det\nabla^2h(0)}\neq 0$. By a corresponding choice of possibly complex eigenvectors $a_{\pm}$ we can conclude the statement.
\end{proof}
This Corollary gives us a hint how the type of the critical point of $h$ influences the spectral stability of the periodic solutions. If for example $0$ is a saddlepoint, then $X_r(2\pi)$ in second order stretches a certain vector $\hat{a}\in D$ by a factor $1+cr^2$ for some $c>0$. However, the existence of such a vector alone does not imply the existence of an actual eigenvalue $\lambda(r)=1+cr^2+o(r^2)$ of $X_r(2\pi)$, cf. Example \ref{ex:counterexample_to_actual_eigenvalues} below. The next section provides a condition, under which such a  conclusion is possible.

\section{Approximately simple eigenvalues}\label{sec:approximately_simple_eigenvalues}

Now we will consider a general eigenvalue and eigenvector continuation problem. Let $V$ be a finite-dimensional vector space over $\C$ equipped with some norm. For a linear map $A\in \cL(V)$ and an eigenvalue $\lambda\in \sigma(A)$ we denote the corresponding eigenspace by $E_\lambda(A)=\ker(A-\lambda\id_V)$ and by $\nu_\lambda(A)$ the algebraic multiplicity of $\lambda$. If $\nu_\lambda(A)=1$, then $\lambda$ is called a simple eigenvalue of $A$.

Consider now a continuous family of linear maps $M:[0,r_0)\rightarrow \cL(V)$, $r\mapsto M_r$.
Suppose that $\lambda_0\in\sigma(M_0)$ is an eigenvalue of $M_0$ with corresponding eigenvector $v_0\in E_{\lambda_0}(M_0)\setminus\{0\}$. It is well known that if $\lambda_0$ is a simple eigenvalue, then there exists $r_1<r_0$ and continuous functions $\lambda:[0,r_1)\rightarrow\C$, $v:[0,r_1)\rightarrow V$ with $\lambda(0)=\lambda_0$, $v(0)=v_0$ and such that $M_rv(r)=\lambda(r)v(r)$ for any $r\in[0,r_1)$. For the convenience we briefly sketch the proof. Let $F:[0,r_0)\times V\times\C\rightarrow V\times\C$,
\[
F(r,v,\lambda)=\big(M_rv-\lambda v, \phi(v)-1\big),
\] 
where $\phi:V\rightarrow \C$ is an arbitrary linear functional with $\phi(v_0)=1$. Then $F(0,v_0,\lambda_0)=0$ and the derivative
\[
D_{(v,\lambda)}F(0,v_0,\lambda_0)[v,\lambda]=\big(M_0v-\lambda_0v-\lambda v_0,\phi(v)\big)
\]
is an isomorphism. Indeed $D_{(v,\lambda)}F(0,v_0,\lambda_0)[v,\lambda]=0$ implies $(M_0-\lambda_0\id_V)^2v=0$ and thus  $v\in\C v_0$ by the simplicity of the eigenvalue. But then $\phi(v)=0$ yields $v=0$ and finally $\lambda=0$. Hence the statement follows by the implicit function theorem. Of course if we would have required $r\mapsto M_r$ to be of class $\cC^k$, then also $r\mapsto \lambda(r)$ and $r\mapsto v(r)$ are $\cC^k$ functions. This settles the case of simple eigenvalues.

In the case that $\lambda_0$ is not a simple eigenvalue, a continuous choice of eigenvalue functions emanating from $\lambda_0$ is still possible. This follows by the continuous dependence of the roots of polynomials with respect to the coefficients. But a continuation of the corresponding eigenvectors is in general, even in the diagonalizable case when $\dim E_{\lambda_0}(M_0)=\nu_{\lambda_0}(M_0)$, not possible. A (symplectic) example is the family
\[
M_0=\begin{pmatrix}
1&0\\
0&1
\end{pmatrix},\quad
M_r=R\left(1/r\right) \begin{pmatrix}
1+r & 0\\
0 & \frac{1}{1+r}
\end{pmatrix} R\left(1/r\right)^{-1}, r>0,
\]
with
\[
R(\theta)=\begin{pmatrix}
\cos \theta & -\sin\theta\\
\sin\theta & \cos\theta
\end{pmatrix}.
\]

In the remaining part of this section we will introduce a condition that allows us to construct continuous eigenvector functions emanating from the eigenspace of a multiple eigenvalue. The condition is based on a higher order approximation.

\begin{definition}\label{def:approx_simple_eigenvalue}
The eigenvalue $\lambda_0\in\sigma(M_0)$ is called approximately simple with respect to the family $(M_r)_{r\in[0,r_0)}$ if the eigenspace $V_0:=E_{\lambda_0}(M_0)$ has dimension $\nu:=\nu_{\lambda_0}(M_0)$
and if there exists a linear map $B_0\in \cL(V_0)$ having $\nu$ distinct eigenvalues, as well as a continuous function $f:[0,r_0)\rightarrow\C$ with $f(0)=0$, $f(r)\neq 0$ for $r\in(0,r_0)$, such that the restriction $M_{r|V_0}:V_0\rightarrow V$ can be written as
\begin{equation}\label{eq:approx_simple_eigenvalue_condition}
M_{r|V_0}=\lambda_0\id_{V_0}+f(r)B_0+o(f(r)),
\end{equation}
when $r\rightarrow 0$.
\end{definition}
Note that if $f(r)=r^n$, condition \eqref{eq:approx_simple_eigenvalue_condition} is a kind of partial differentiability.

The following lemma shows that the approximate eigenvalues $\lambda_0+f(r)\mu_0$, $\mu_0\in\sigma(B_0)$ give rise to actual eigenvalues.  
\begin{lemma}\label{lem:actual_eigenvalues}
Let $\lambda_0\in\sigma(M_0)$ be an approximately simple eigenvalue of $M_0$ with associated maps $B_0$ and $f$ as in Definition \ref{def:approx_simple_eigenvalue}. Then for every $\mu_0\in\sigma(B_0)$ and $v_0\in E_{\mu_0}(B_0)\setminus\{0\}$ there exist $r_1\in(0,r_0)$ and continuous maps $\lambda:[0,r_1)\rightarrow \C$, $v:[0,r_1)\rightarrow V$, $v_1:[0,r_1)\rightarrow E_{\lambda_0}(M_0)$ with 
\begin{gather*}
M_rv(r)=\lambda(r)v(r),\quad \lambda(r)=\lambda_0+f(r)\mu_0+o(f(r))\\
v(r)=v_1(r)+o(f(r)),\quad v(0)=v_1(0)=v_0.
\end{gather*}
\end{lemma}
Lemma \ref{lem:actual_eigenvalues} will be applied in our stability analysis. Note that the existence of a single vector $e\in V$ satisfying $M_re=(\lambda_0+f(r))e+o(f(r))$ is not sufficient to conclude that an actual eigenvalue $\lambda(r)=\lambda_0+f(r)+o(f(r))$ exists. We provide the following symplectic counterexample.
\begin{example}\label{ex:counterexample_to_actual_eigenvalues}
For $r\geq 0$ there holds
\[
M_re_1:=\begin{pmatrix}
\cos(r^2)+\sin(r^2) & -2r\\
\frac{\sin^2(r^2)}{r}&\cos(r^2)-\sin(r^2)
\end{pmatrix}\begin{pmatrix}
1\\0
\end{pmatrix}=(1+r^2)e_1+o(r^2),
\]
but the computation of the eigenvalues via the characteristic polynomial 
\[
\lambda^2-2\cos(r^2)\lambda+1=0
\] 
shows $\sigma(M_r)=\set{e^{ir^2},e^{-ir^2}}$ and $e^{\pm i r^2}=1\pm ir^2+o(r^2)$.
\end{example}
\begin{proof}[Proof of Lemma \ref{lem:actual_eigenvalues}] As in the case of a simple eigenvalue, the proof is based on the implicit function theorem.
For $\tilde{v}\in V\setminus\{0\}$, $\lambda\in \C$, $r>0$ we want to solve the equation
\begin{equation}\label{eq:eigenvalue_equation}
M_r\tilde{v}-\lambda \tilde{v}=0.
\end{equation}
In order to do this we fix a complement $W$ of $V_0=E_{\lambda_0}(M_0)$, i.e., $V=V_0\oplus W$ and write $\lambda=\lambda_0+f(r)\mu$, $\mu\in\C$, as well as
$\tilde{v}=v+f(r)w$, $v\in V_0$, $w\in W$. Plugging this ansatz into \eqref{eq:eigenvalue_equation} and dividing by $f(r)$ we obtain the equivalent equation
\begin{align}\label{eq:equivalent_formulation_after_ansatz}
F_1(r,v,w,\mu):=\frac{M_rv-\lambda_0v}{f(r)}-\mu v+M_rw-\lambda_0w-f(r)\mu w=0.
\end{align}
By \eqref{eq:approx_simple_eigenvalue_condition} the map $F_1:(0,r_0)\times V_0\times W\times \C\rightarrow V$ continuously extends to
\[
F_1(0,v,w,\mu)=B_0v-\mu v+M_0w-\lambda_0w
\]
as $r\rightarrow 0$. The same holds true for the partial derivative $D_{(v,w,\mu)}F_1$ with respect to $v,w$ and $\mu$.

Let us now fix $\mu_0\in \sigma(B_0)$, a corresponding eigenvector $v_0\in E_{\mu_0}(B_0)\subset V_0$ and an arbitrary linear functional $\phi:V_0\rightarrow\C$ satisfying $\phi(v_0)=1$. Consider the map $F:[0,r_0)\times V_0\times W\times\C\rightarrow V\times\C$,
\begin{align*}
F(r,v,w,\mu)=\big(F_1(r,v,w,\mu),\phi(v)-1\big).
\end{align*}

Then we have $F(0,v_0,0,\mu_0)=0$. Furthermore, at $(0,v_0,0,\mu_0)$ the derivative with respect to $(v,w,\mu)$ is given by
\[
D_{(v,w,\mu)}F(0,v_0,0,\mu_0)[v,w,\mu]=\big(B_0v-\mu_0 v-\mu v_0+M_0w-\lambda_0w,\phi(v)\big).
\]
Thus for $(v,w,\mu)\in\ker D_{(v,w,\mu)}F(0,v_0,0,\mu_0)$ there holds $\phi(v)=0$ and
\begin{equation}\label{eq:kernel_equation}
B_0v-\mu_0v-\mu v_0=-(M_0w-\lambda_0w).
\end{equation} 
Since the left-hand side of \eqref{eq:kernel_equation} is contained in $V_0$, we get $(M_0-\lambda_0\id_V)^2w=0$, which shows $w\in V_0$, since the algebraic multiplicity $\nu_{\lambda_0}(M_0)$ coincides by Definition \ref{def:approx_simple_eigenvalue} with the geometric multiplicity $\dim E_{\lambda_0}(M_0)$. But then $w\in V_0\cap W=\{0\}$. Therefore \eqref{eq:kernel_equation} reduces to $B_0v-\mu_0v=\mu v_0$, which implies $(B_0-\mu_0\id_{V_0})^2v=0$. As before, since $\mu_0$ is a simple eigenvalue of $B_0$, $v\in\C v_0$. By $\phi(v)=0$ and $\phi(v_0)=1$, $v=0$ follows. Finally \eqref{eq:kernel_equation} shows $\mu=0$.

Therefore $D_{(v,w,\mu)}F(0,v_0,0,\mu_0)$ is an isomorphism and the implicit function theorem provides solutions $v(r)\in V_0$, $w(r)\in W$, $\mu(r)\in \C$ of $F_1(r,\cdot,\cdot,\cdot)=0$. By our ansatz $\lambda(r)=\lambda_0+f(r)\mu(r)$ and $\tilde{v}(r)=v(r)+f(r)w(r)$ are the desired eigenvalue and eigenvector functions.
\end{proof}

Lemma \ref{lem:actual_eigenvalues} or rather the bifurcation of a multiple eigenvalue in general is related to Theorem 7 of Lancaster's paper \cite{lancaster_eigenvalues_1964}, but note that contrary to \cite{lancaster_eigenvalues_1964} the family of matrices here does not need to depend analytically on the parameter.
\section{Application to a Poincar\'e section}\label{sec:application_to_a_poincare_section}
Suppose from now on that $g$ is of class $\cC^m$ with $m\geq 4$. By Lemma \ref{lem:expansion_of_monodromy_operator_on_D} and Corollary \ref{cor:almost_eigenvalues_for_X_r} the restriction of the monodromy operator $X_r(2\pi)_{|D}:D\rightarrow \R^{2N}$ has the structure
\[
X_r(2\pi)_{|D}=\id_D-2\pi\Gamma r^2 B_0+o(r^2),
\]
with $B_0:D\rightarrow D$, $B_0\hat{a}=\big(J\nabla^2h(0)a\big)\widehat{\phantom{|}}$ and $\sigma(B_0)=\set{\pm\sqrt{-\det\nabla^2h(0)}}$. Nonetheless Lemma \ref{lem:actual_eigenvalues} still does not apply due to the fact that the eigenspace $E_1(X_0(2\pi))$ is bigger than $D$. In fact by the definition of a nondegenerate relative equilibrium the geometric multiplicity is $3$ and the algebraic is at least $4$. 
We suppose from now on that the algebraic multiplicity of $1$ is exactly $4$, i.e., that $Z$ is algebraic nondegenerate. This way the generalized eigenspace of $X_0(2\pi)$ is precisely given by $D\oplus \R J_Nz_0\oplus\R z_0$. 
A suitable Poincar\'e section will reduce the generalized eigenspace to the space $D$ only, such that Lemma \ref{lem:actual_eigenvalues} can be applied.

\subsection{A linear section}
Let $(r,t,u)\mapsto \phi_r(t,u)$ denote the flow of \eqref{eq:n_vortex_with_parameter}, which is defined on an open subset of $[0,r_0)\times\R\times\R^{2N}$ and of class $\cC^{m-1}$. Then the relative equilibrium solution can be written as $Z(t)=\phi_0(t,z_0)$, more generally $u^{(r)}(t)=\phi_r\big(t,u^{(r)}(0)\big)$ and the monodromy matrix can be expressed by $X_r(2\pi)=D_u\phi_r\big(2\pi,u^{(r)}(0)\big)$. Recall also that \eqref{eq:n_vortex_with_parameter} is Hamiltonian with respect to the non-standard symplectic form $\omega_{\vec{\Gamma}}(v,w)=\ska{M_{\vec{\Gamma}} v,J_N w}$. Thus $X_r(2\pi)$ is a symplectic linear mapping, i.e., $\omega_{\vec{\Gamma}}(X_r(2\pi)\cdot,X_r(2\pi)\cdot)=\omega_{\vec{\Gamma}}$.  
We consider the linear subspace
\[
\Sigma=\set{z\in\R^{2N}:\omega_{\vec{\Gamma}}(z_0,z)=0}
\]
and apply the implicit function theorem to the equation $\tilde{\omega}(r,t,u)=0$ where $\tilde{\omega}(r,t,u)=\omega_{\vec{\Gamma}}(z_0,\phi_r(t,u))$. There holds $\tilde{\omega}(0,2\pi,z_0)=0$ and
\begin{align*}
\partial_t\tilde{\omega}(0,2\pi,z_0)&=\frac{d}{dt}_{|t=2\pi}\omega_{\vec{\Gamma}}(z_0,\phi_0(t,z_0))=\omega_{\vec{\Gamma}}(z_0,\dot{Z}(2\pi))\\&=\omega_{\vec{\Gamma}}(z_0,J_NM_{\vec{\Gamma}}^{-1}\nabla H_0(z_0))=\frac{L}{\pi}\neq 0
\end{align*}
by \eqref{eq:lem:properties_of_whole_plane:nabla_H_times_z} and Lemma \ref{lem:algebraic_nondegenerate_implies_L_not_0}. Thus the implicit function theorem provides a hitting time $\tau\in\cC^{m-1}\big([0,r_1)\times B_{\varepsilon_0}(z_0)\big)$ for some $r_1\in(0,r_0)$ and $\varepsilon_0>0$ satisfying $\phi_r(\tau(r,u),u)\in\Sigma$ for all $(r,u)\in[0,r_1)\times B_{\varepsilon_0}(z_0)$. In particular $\tau(0,z_0)=2\pi$. Since $u^{(r)}(0)\rightarrow z_0$ as $r\rightarrow 0$, $\tau_r:=\tau\big(r,u^{(r)}(0)\big)$ is well-defined on a subinterval $[0,r_2)\subset[0,r_1)$ and 
\[
u^{(r)}\left(\tau_r\right)=\phi_r\left(\tau_r,u^{(r)}(0)\right)\in\Sigma
\]
for all $r\in[0,r_2)$.
\begin{lemma}\label{lem:derivative_of_tau_r_0}
The map $r\mapsto\tau_r$ is contained in $\cC^{m-2}([0,r_2))\cap\cC^{m-1}((0,r_2))$ with $\tau_0=2\pi$ and $\partial_r\tau_0=0$.
\end{lemma}
\begin{proof} The regularity of $r\mapsto \tau_r$ follows from the regularity of $\tau$ and $r\mapsto u^{(r)}(0)$. Also by definition $\tau_0=\tau(0,z_0)=2\pi$. For the derivative $\partial_r\tau_0$ we use the $2\pi$-periodicity of $u^{(r)}$ and differentiate the defining equation 
\[
\omega_{\vec{\Gamma}}\left(z_0,\phi_r\left(\tau_r-2\pi,u^{(r)}(0)\right)\right)=0
\]
at $r=0$. This yields
\begin{align*}
0=\omega_{\vec{\Gamma}}\left(z_0,\partial_r\phi_0(0,z_0)+\partial_t\phi_0(0,z_0)\partial_r\tau_0+D_u\phi_0(0,z_0)\partial_ru^{(0)}(0)\right).
\end{align*}
Now $\phi_r(0,u)=u$ implies $\partial_r\phi_0(0,z_0)=0$ and $D_u\phi(0,z_0)=\id_{\R^{2N}}$. As before, $\partial_t\phi_0(0,z_0)=M_{\vec{\Gamma}}^{-1}J_N\nabla H_0(z_0)$. Thus \eqref{eq:lem:properties_of_whole_plane:nabla_H_times_z}, \eqref{eq:lem:properties_of_whole_plane:center_of_vorticity_0}  and \eqref{eq:derivative_of_u_r} show
\[
0=\omega_{\vec{\Gamma}}\big(z_0,M_{\vec{\Gamma}}^{-1}J_N\nabla H_0(z_0)\big)\partial_r\tau_0+\omega_{\vec{\Gamma}}\big(z_0,\partial_ru^{(0)}\big)=\frac{L}{\pi}\partial_r\tau_0,
\]
which finishes the proof of the Lemma, since $L\neq 0$ by the algebraic nondegenerateness of $Z$, cf. Lemma \ref{lem:algebraic_nondegenerate_implies_L_not_0}.
\end{proof}
By Lemma \ref{lem:admissible_time_shift} we can without restriction assume that
\[
u_r:=u^{(r)}(0)\in\Sigma,\quad \tau(r,u_r)=2\pi
\]
for all $r\in[0,r_2)$. More precisely, if we pass to $\tilde{u}^{(r)}=u^{(r)}(\cdot+\tau_r)$, then $\tilde{u}^{(r)}(t)$ and the associated monodromy operator $\tilde{X}_r(2\pi)$ has the same properties as $u^{(r)}$ and $X_r(t)$ and additionally there holds $\tilde{u}^{(r)}(0)\in\Sigma$. We therefore can directly assume that $u^{(r)}=\tilde{u}^{(r)}$.
\begin{lemma}\label{lem:properties_of_tau}
For $v\in \nabla H_0(z_0)^\perp\cap \Sigma$, $\hat{a}\in D$ as $r\rightarrow 0$ there holds
\[
D_u\tau(0,z_0)v=0,\quad D_u\tau(r,u_r)\hat{a}=o(r^2),\quad D_u\tau(r,u_r)z_0=4\pi+O(r^2).
\]
\end{lemma}
\begin{proof}
Differentiation of $\omega_{\vec{\Gamma}}\big(z_0,\phi_r(\tau(r,u),u)\big)=0$ at $u=u_r$ in direction $v\in\R^{2N}$ shows
\begin{equation}\label{eq:differentiation_for_derivative_of_tau}
0=\omega_{\vec{\Gamma}}(z_0,X_r(2\pi)v)+\left(\frac{L}{\pi}+o(1)\right)D_u\tau(r,u_r)v
\end{equation}
as $r\rightarrow 0$.
In particular for $v=\hat{a}\in D$ we conclude by \eqref{eq:lem:properties_of_whole_plane:center_of_vorticity_0} and Lemma \ref{lem:expansion_of_monodromy_operator_on_D} that $D_u\tau(r,u_r)\hat{a}=o(r^2)$. By the general expansion in Lemma \ref{lem:expansion_of_monodromy_operator_on_D} and by Lemma \ref{lem:floquet_solutions_of_whole_plane_case} we see 
\begin{align*}
\omega_{\vec{\Gamma}}(z_0,X_r(2\pi)z_0)&=\omega_{\vec{\Gamma}}(z_0,z_0+4\pi \nu J_Nz_0+O(r^2))=-4\pi\nu\ska{M_{\vec{\Gamma}} z_0,z_0}+O(r^2)\\
&=-4\pi\nu\ska{-\nu\nabla H_0(z_0),z_0}+O(r^2)=-4L+O(r^2).
\end{align*}
In the second to last step we simply used that $Z(t)$ is a solution of the whole-plane system \eqref{eq:n_vortex_whole_plane}. The expansion $D_u\tau(r,u_r)z_0=4\pi+O(r^2)$ follows.

It remains to look at \eqref{eq:differentiation_for_derivative_of_tau} for $r=0$ and $v\in \nabla H_0(z_0)^\perp\cap \Sigma$. We use that $X_0(2\pi)$ belongs to the group of symplectic matrices with respect to $\omega_{\vec{\Gamma}}$ and conclude
\begin{align*}
-\frac{L}{\pi}D_u\tau(0,z_0)v&=\omega_{\vec{\Gamma}}(X_0(2\pi)^{-1}z_0,v)=\omega_{\vec{\Gamma}}(z_0-4\pi \nu J_N z_0,v)\\
&=-4\pi\nu\ska{M_{\vec{\Gamma}} z_0,v}=4\pi\ska{\nabla H_0(z_0),v}=0.
\end{align*}
\end{proof}

\subsection{Restriction to energy levels}
Next we restrict ourselves to corresponding energy levels. Let
\[
\Sigma_r=\Sigma\cap H_r^{-1}(H_r(u_r)),\quad\Sigma^\varepsilon_r=\Sigma_r\cap B_\varepsilon(z_0),\quad r\in[0,r_2),~\varepsilon>0.
\]
Since the flow $\phi_r$ preserves energy levels, we can define the Poincar\'e return map $P_r:\Sigma^{\varepsilon_0}_r\rightarrow\Sigma_r$,
\[
P_r(u)=\phi_r(\tau(r,u),u)
\]
for any $r\in[0,r_2)$. 
\begin{lemma}\label{lem:poincare_map_symplectic}
There exists $r_3\in(0,r_2)$ and $0<\varepsilon_2<\varepsilon_1<\varepsilon_0$, such that $\big(\Sigma^{\varepsilon_1}_r,\omega_{\vec{\Gamma}}\big)$ is a symplectic $\cC^m$-submanifold, $P_r:\Sigma^{\varepsilon_2}_r\rightarrow\Sigma^{\varepsilon_1}_r$ is a symplectic $\cC^{m-1}$-map, $u_r\in \Sigma^{\varepsilon_2}_r$ and $P_r(u_r)=u_r$.
\end{lemma}
The proof is the same as in the parameter independent case, which can be found for example in \cite{moser_notes_2005}, Theorem 2.5.
\begin{lemma}\label{lem:floquet_multipliers_relation_to_monodromy_matrix}
Counting the algebraic multiplicity the periodic solution $u^{(r)}$ has the Floquet multipliers $1,1,\lambda_2,\lambda_2^{-1},\ldots,\lambda_N,\lambda_N^{-1}$, if and only if the linearization of the Poincar\'e map $DP_r(u_r):T_{u_r}\Sigma_r\rightarrow T_{u_r}\Sigma_r$ has the spectrum $\lambda_2,\lambda_2^{-1},\ldots,\lambda_N,\lambda_N^{-1}$.
\end{lemma}
For a proof we refer to \cite{meyer_introduction_2009}, Lemma 8.5.6.

\subsection{A common system of coordinates}
We want to apply the notion of approximately simple eigenvalues to the family $(DP_r(u_r))_{r\in[0,r_3)}$. Unfortunately these linear maps are not defined on a fixed linear space as it is required in Definition \ref{def:approx_simple_eigenvalue}. Moreover, for $r>0$ the space $D$, on which we know the expansion of $X_r(2\pi)$, is not contained in the domain $T_{u_r}\Sigma_r$ of $DP_r(u_r)$. In order to get around this we introduce a common system of coordinates using the tangent space
\[
T_{z_0}\Sigma_0=\set{u\in\R^{2N}:\omega_{\vec{\Gamma}}(z_0,u)=0,~\ska{\nabla H_0(z_0),u}=0}.
\]

\begin{lemma}\label{lem:existence_of_common_chart}
There exists $U\subset T_{z_0}\Sigma_0$ open, $0\in U$, $r_4\in(0,r_3)$ and a map $s\in\cC^m\big([0,r_4)\times U\big)$, $(r,u)\mapsto s(r,u)$ with $s(0,0)=1$, $D_us(0,0)v=0$ for every $v\in T_{z_0}\Sigma_0$ and such that for $r\in[0,r_4)$ the map
$\psi_r: U\rightarrow \Sigma_r$,
\[
\psi_r(u)=u+s(r,u)z_0
\]
is a symplectic $\cC^m$-diffeomorphism onto its image $\psi_r(U)\subset \Sigma_r$. The inverse transformation $\psi_r^{-1}:\psi_r(U)\rightarrow U$ is given by
\[
\psi_r^{-1}(w)=w+\frac{\pi}{L}\ska{\nabla H_0(z_0),w}z_0
\]
Moreover, $r_4>0$ can be choosen in a way such that $u_r\in\psi_r(U)$ for $r\in[0,r_4)$.
\end{lemma}
\begin{proof}
Observe that if $u\in T_{z_0}\Sigma_0\subset \Sigma$, then $u+sz_0\in\Sigma$ for any $s\in\R$. It therefore remains to find a suitable choice of $s$, such that $u+s(r,u)z_0$ is also contained in the energy level belonging to $u_r$. In other words we need to solve $H_r(u+sz_0)=H_r(u_r)$ with respect to $s$. Since the equation obviously holds for $r=0$, $u=0$, $s=1$ and since $\ska{\nabla H_0(z_0),z_0}=-\frac{L}{\pi}\neq 0$, the implicit function theorem gives us numbers $r_4>0$, $\delta>0$, as well as a neighborhood $U\subset T_{z_0}\Sigma_0$ of $0$ and $s:[0,r_4)\times U\rightarrow (1-\delta,1+\delta)$, such that for $(r,u,s)\in[0,r_4)\times U\times(1-\delta,1+\delta)$ there holds $u+sz_0\in\Sigma_r$, if and only if $s=s(r,u)$. Clearly $s(0,0)=1$ and differentiation of $H_r(u+s(r,u)z_0)=H_r(u_r)$ shows
\[
D_us(0,0)v=\frac{\pi}{L}\ska{\nabla H_0(z_0),v}.
\]
Hence $D_us(0,0)v=0$ for $v\in T_{z_0}\Sigma_0\subset \nabla H_0(z_0)^\perp$.

Now for $r\in[0,r_4)$ define $\psi_r:U\rightarrow\Sigma_r$, $\psi_r(u)=u+s(r,u)z_0$. Via the splitting $\Sigma=T_{z_0}\Sigma_0\oplus\R z_0$, the set $O=\set{u+sz_0:u\in U,~s\in(1-\delta,1+\delta)}$ is open in $\Sigma$ and therefore $O\cap\Sigma_r=\psi_r(U)$ is open in $\Sigma_r$. 

Next one easily checks by Lemma \ref{lem:properties_of_whole_plane_hamiltonian} that $\tilde{\psi}_r:\Sigma\rightarrow T_{z_0}\Sigma_0$, 
\[
\tilde{\psi}_r(w)=w+\frac{\pi}{L}\ska{\nabla H_0(z_0),w}z_0
\] is well-defined as a mapping into $T_{z_0}\Sigma_0$ and that $\tilde{\psi}_r(\psi_r(u))=u$ for $u\in U$. It follows $\tilde{\psi}_r(\psi_r(U))=U$ and also $\psi_r(\tilde{\psi}_r(w))=w$ for $w\in\psi_r(U)$. Hence $\psi_r:U\rightarrow\psi_r(U)$ is a diffeomorphism with $\psi_r^{-1}=\tilde{\psi}_{r|\psi_r(U)}$.
That $\psi_r$ is even a symplectic diffeomorphism follows directly from the definition of $\Sigma$ and $\psi_r$.

Finally, $\tilde{\psi_r}(u_r)\rightarrow 0\in U$ as $r\rightarrow 0$ implies $u_r\in \psi_r(U)$ for $r$ sufficiently small.
\end{proof}

\subsection{\texorpdfstring{The bifurcation of Floquet multiplier $1$}{The bifurcation of Floquet multiplier 1}}
Instead of $DP_r(u_r):T_{u_r}\Sigma_r\rightarrow T_{u_r}\Sigma_r$ we will now study the family
\[
M_r:=D\psi_r^{-1}(u_r)\circ DP_r(u_r)\circ D\psi_r(\psi_r^{-1}(u_r)):T_{z_0}\Sigma_0\rightarrow T_{z_0}\Sigma_0.
\]
In particular we will show that $1\in\sigma(M_0)$ is an approximately simple eigenvalue with respect to the $\cC^{m-2}$-family $(M_r)_{r\in[0,r_4)}$ of symplectic maps. By the definition of $M_r$ and by Lemma \ref{lem:floquet_multipliers_relation_to_monodromy_matrix} the spectrum of $M_0$ is given by $1,1,\lambda_3,\lambda_3^{-1},\ldots,\lambda_N,\lambda_N^{-1}$, where $\lambda_j,\lambda_j^{-1}\neq 1$, $j=3,\ldots,N$ are the nontrivial Floquet multipliers of the algebraic nondegenerate relative equilibrium $Z(t)$. Due to the symplectic nature of $M_r$, the pair $1,1$ can only bifurcate into a pair $\lambda_2(r),\lambda_2(r)^{-1}\in \R\cup S^1$. For a bifurcation into $\C\setminus(\R\cup S^1)$ the mulitiplicity of the eigenvalue $1$ of $M_0$ has to be at least four, since $\lambda_2(r),\lambda_2(r)^{-1},\bar{\lambda}_2(r),\bar{\lambda}_2(r)^{-1}$ would be four different eigenvalues of $M_r$, $r>0$.
\begin{lemma}\label{lem:1_is_approximately_simple}
The eigenspace $E_1(M_0)$ is given by $E_1(M_0)=D\subset T_{z_0}\Sigma_0$. Furthermore, as $r\rightarrow 0$ we have
\[
M_{r|D}=\id_D-2\pi\Gamma r^2 B_0+o(r^2),
\]
where $B_0:D\rightarrow D$, $B_0\hat{a}=\big(J\nabla^2h(0)a\big)\widehat{\phantom{|}}$ and $\sigma(B_0)=\set{\pm\sqrt{-\det\nabla^2h(0)}}$.
In other words the eigenvalue $1\in\sigma(M_0)$ is approximately simple with respect to the family $(M_r)_{r\in[0,r_4)}$.
\end{lemma}
\begin{proof}
Since $D\perp \nabla H_0(z_0)$ and $M_{\vec{\Gamma}} z_0\perp D$ by Lemma \ref{lem:properties_of_whole_plane_hamiltonian}, $D$ is contained in the tangent space $T_{z_0}\Sigma_0$. 

By Lemma \ref{lem:existence_of_common_chart} $D\psi_0(0)=D\psi_0(\psi_0^{-1}(z_0)):T_{z_0}\Sigma_0\rightarrow T_{z_0}\Sigma_0$ is the identity and thus $M_0=DP_0(z_0)$. Next Lemma \ref{lem:properties_of_tau} shows that $M_0=DP_0(z_0)=X_0(2\pi)_{|T_{z_0}\Sigma_0}$ and we conclude $E_1(M_0)=D$.

Now we begin with the expansion of $M_r$ on $D$. Let $x_r:=\psi_r^{-1}(u_r)$, $\hat{a}\in D$, $\hat{b}_r:=\hat{a}-2\pi\Gamma r^2 B_0\hat{a}\in D$ and consider 
\begin{align*}
D\psi_r(x_r)M_r\hat{a}&=DP_r(u_r)D\psi_r(x_r)\hat{a}\\
&=\big(D_u\phi_r(2\pi,u_r)+\partial_t\phi_r(2\pi,u_r)D_u\tau(r,u_r)\big)\big[\hat{a}+z_0 D_u s(r,x_r)\hat{a}\big]\\
&=X_r(2\pi)\hat{a}+D_us(r,x_r)[\hat{a}]X_r(2\pi)z_0+\alpha(r,\hat{a})M_{\vec{\Gamma}}^{-1}J_N\nabla H_r(u_r),
\end{align*}
where
\begin{align*}
\alpha(r,\hat{a})&=D_u\tau(r,u_r)\hat{a}+D_u\tau(r,u_r)[z_0]D_us(r,x_r)[\hat{a}]\\
&=o(r^2)+(4\pi+O(r^2))D_us(r,x_r)\hat{a}=4\pi D_us(r,x_r)\hat{a}+o(r^2)
\end{align*}
as $r\rightarrow 0$ by Lemma \ref{lem:properties_of_tau}. In particular by Lemma \ref{lem:existence_of_common_chart}, $\alpha(r,\hat{a})=o(1)$. Using now this expansion for $\alpha(r,\hat{a})$, the expansion of the monodromy operator in Lemma \ref{lem:expansion_of_monodromy_operator_on_D}, as well as Lemma \ref{lem:floquet_solutions_of_whole_plane_case} and \eqref{eq:lem:properties_of_whole_plane:orthogonality_of_H_0},\eqref{eq:lem:basic_properties_derivatives_of_H_r_0} it follows
\begin{align*}
D\psi_r(x_r)M_r\hat{a}&=\hat{b}_r+o(r^2)+\big(z_0+4\pi\nu J_N z_0+O(r^2)\big)D_us(r,x_r)[\hat{a}]\\
&\hspace{33pt}+\alpha(r,\hat{a})\big(-\nu J_Nz_0+O(r^2)\big)\\
&=\hat{b}_r+z_0D_us(r,x_r)[\hat{b}_r]+z_0D_us(r,x_r)[\hat{a}-\hat{b}_r]\\
&\hspace{33pt}+\nu J_N z_0\big(4\pi D_us(r,x_r)\hat{a}-\alpha(r,\hat{a})\big)+o(r^2)\\
&=D\psi_r(x_r)\hat{b}_r+\nu J_N z_0\big(4\pi D_us(r,x_r)\hat{a}-\alpha(r,\hat{a})\big)+o(r^2)\\
&=D\psi_r(x_r)\hat{b}_r+o(r^2).
\end{align*}
Therefore we conclude
\[
M_r\hat{a}=\hat{a}-2\pi\Gamma r^2 B_0\hat{a}+o(r^2).
\]
\end{proof}
\begin{proof}[Proof of Theorem \ref{thm:linear_stability}]
Let $\Gamma$, $g$, $a_0$, $Z(t)$ be as stated in the Theorem, without restriction $a_0=0$. 
Combining Lemma \ref{lem:actual_eigenvalues} and Lemma \ref{lem:1_is_approximately_simple} we see that $u^{(r)}$ for $r>0$ small has a pair of Floquet multipliers
\[
\lambda_{\pm}(r)=1\pm 2\pi\Gamma r^2\sqrt{-\det \nabla^2 h(0)}+o(r^2).
\]
Therefore if $a_0=0$ is a saddle point of $h$, then $\lambda_{\pm}(r)\in\R\setminus\{1\}$ and thus $u^{(r)}$ is spectrally unstable. This 
proves part a). 

If $a_0=0$ is a minimum or maximum of $h$, then we see that the double eigenvalue $1,1$ of $M_0$ bifurcates into a pair $\lambda_{\pm}(r)$ of simple eigenvalues lying in $S^1$. If additionally the remaining nontrivial multipliers $\lambda_3,\lambda_3^{-1},\ldots,\lambda_N,\lambda_N^{-1}$ of $Z(t)$ are all simple and contained in $S^1$, i.e., if $Z(t)$ is LRE-stable, then these multipliers can not leave $S^1$ unless they collide. Therefore for all $r>0$ small enough the nontrivial multipliers of $u^{(r)}$ are all contained in $S^1$ and simple. Thus $u^{(r)}$ is L-stable.
\end{proof}
\begin{remark}
The proof of Theorem \ref{thm:linear_stability} shows that the requirement of $Z(t)$ being LRE-stable is not necessary for the instability result in part a). The algebraic nondegenerateness of $Z(t)$ is enough to conclude that the induced solutions are spectrally unstable.
\end{remark}

\section{Nonlinear stability for two vortices}\label{sec:nonlinear_stability}
In order to obtain a nonlinear stability result for $N=2$ vortices we follow ideas of the course \cite{ortega_stability_2017}. With this approach the computation of the twist coefficient, for example with the formulas given in \cite{ortega_stability_2017}, is not necessary. The price to pay is the loss of a set of parameter values having measure $0$.
\begin{proof}[Proof of Theorem \ref{thm:nonlinear_stability_2_vortices}]
Let $g\in\cC^\infty(\Omega\times\Omega)$, $N=2$ and $\Gamma_1+\Gamma_2\neq 0$. Assume that $a_0=0\in\Omega$ is a nondegenerate local minimum or maximum of $h$ and let $\big(u^{(r)}\big)_{r\in[0,r_0)}$ be the solutions of Proposition \ref{prop:existence_result_detailled} with $u^{(0)}$ being the $2\pi$-periodic vortex pair of the whole-plane system \eqref{eq:n_vortex_whole_plane}. Note that now $[0,r_0)\times\R\ni(r,t)\mapsto u^{(r)}(t)\in\R^{4}$ is of class $\cC^\infty$ and that any Poincar\'e map associated to $u^{(r)}$ is defined on a $2$-dimensional symplectic submanifold of $\R^4$, which without restriction we can assume to be smooth.

By our linear analysis in Section \ref{sec:application_to_a_poincare_section} we know that any such Poincar\'e map can be conjugated in first order to a rotation by the angle $\varepsilon(r)$, where
\begin{equation}\label{eq:nontrivial_floquet_multipliers_of_2_vortices}
e^{\pm i\varepsilon(r)}=\lambda_{\pm}(r)=1\pm 2\pi\Gamma r^2\sqrt{-\det \nabla^2 h(0)}+o(r^2)
\end{equation}
is the pair of nontrivial Floquet multipliers of $u^{(r)}$. As a consequence of Herman's last theorem, see  Thm. 3 and Section 1.3 in \cite{fayad_herman_2009}, we see that $u^{(r)}$ is isoenergetically orbitally stable provided the rotation number $\frac{\varepsilon(r)}{2\pi}$ satisfies a Diophantine condition. Note that the Poincar\'e map $P_r$ has the intersection property around the fixed point $u_r(0)$ because it is canonical. The Diophantine condition is satisfied, if there exists numbers $\gamma>0$, $\sigma\geq 2$, such that
$\varepsilon(r)\in 2\pi \DC(\gamma,\sigma)$, where
\[
\DC(\gamma,\sigma)=\set{x\in\R: \abs{x-\frac{p}{q}}\geq\frac{\gamma}{q^\sigma} \text{ for all } \frac{p}{q}\in\Q,~q\geq 1}.
\]
The set of all numbers satisfying this condition,
\[
\DC=\bigcup_{\sigma\geq 2}\bigcup_{\gamma>0}\DC(\gamma,\sigma)
\]
has full measure.

Assume for instance that $\varepsilon(r)>0$ for all $r\in(0,r_0)$ and define the diffeomorphism $f:\big(0,\frac{1}{2}\big)\rightarrow (0,4)$, $f(x)=2+2\cos(2\pi x)$. Then $\varepsilon(r)\in 2\pi \DC\cap (0,\pi)$, if and only if $f\big(\frac{\varepsilon(r)}{2\pi}\big)\in f\big(\DC\cap\big(0,\frac{1}{2}\big)\big)=:\DC^\prime$. Now observe that the set $\DC^\prime\subset (0,4)$ has full measure as well and 
\[
f\left(\frac{\varepsilon(r)}{2\pi}\right)=2+2\cos(\varepsilon(r))=2+\lambda_+(r)+\lambda_-(r)=\tr X_r(2\pi).
\]
Thus it remains to show that there exists $r_1\in(0,r_0)$, such that the trace of the fundamental solution $\tr X_r(2\pi)$ lies in $\DC^\prime$ for almost every $r\in(0,r_1)$. 

We abbreviate $c:= 2\pi\Gamma\sqrt{\det\nabla^2h(0)}\in\R\setminus\{0\}$ and write $\lambda_+(r)=\xi(r)+\eta(r)i$ with real valued functions $\xi$ and $\eta$. By equation \eqref{eq:nontrivial_floquet_multipliers_of_2_vortices}, we have $\eta(r)=cr^2+o(r^2)$ and thus 
\[
\xi(r)=\sqrt{1-\eta(r)^2}=1-\frac{1}{2}\eta(r)^2+o(\eta(r)^2)=1-\frac{1}{2}c^2r^4+o(r^4).
\] It follows
\[
\tr X_r(2\pi)=2+2\xi(r)=4-c^2r^4+o(r^4)
\]
Since $T:[0,r_0)\rightarrow (0,4)$, $r\mapsto \tr X_r(2\pi)$ is a smooth map, we conclude that $T$ restricted to some small interval $(0,r_1)$ is a diffeomorphism onto the corresponding image. Therefore $\set{r\in(0,r_1):T(r)\in \DC^\prime}$ has full measure and the proof of Theorem \ref{thm:nonlinear_stability_2_vortices} is finished.
\end{proof}

\vspace{5pt}
\noindent\textbf{Acknowledgements.} B.G. is very grateful for the hospitality of R.O. and the group of Differential Equations during his research stay in Granada.

\vspace{10pt}
\noindent Bj\"orn Gebhard\\
 Mathematisches Institut,  Universit\"at Gie\ss en,  Arndtstr.\ 2,  35392 Gie\ss en,  Germany\\
 bjoern.gebhard@math.uni-giessen.de

\vspace{10pt}
\noindent Rafael Ortega\\ 
 Departamento de Matem\'atica Aplicada, Universidad de Granada, 18071 Granada, Spain\\
 rortega@ugr.es
\end{document}